\pdfoutput=1

\documentclass[12pt]{article}
\usepackage[utf8]{inputenc}
\usepackage{amsmath, amsthm, amssymb}
\usepackage[table,xcdraw]{xcolor}
\usepackage{natbib}
\usepackage[vmargin=3cm, hmargin=2.7cm]{geometry}
\usepackage{pgfplots}
\usepackage{authblk}
\usepackage{chngpage}

\usepackage{commands}
\usepackage{commandsRCPSP}

\newcommand{\barM}{\overline{M}}

\newtheorem{prop}{Proposition}
\newtheorem{definition}{Definition}
\newtheorem{theorem}{Theorem}
\newtheorem{observation}{Observation}
\newtheorem{corollary}{Corollary}

\title{Exact and heuristic methods for Anchor-Robust and Adjustable-Robust RCPSP}

\author[1,2]{Ad\`ele Pass-Lanneau\thanks{Corresponding author.}}
\author[1,2]{Pascale Bendotti}
\author[1,2]{Luca Brunod-Indrigo}

\date{November 3, 2020}

\affil[1]{ 
\normalsize
EDF R\&D, F-91120 Palaiseau, France
}
\affil[2]{ Sorbonne Universit\'e, CNRS, LIP6, F-75005 Paris, France
\hspace{5cm}
\texttt{\{adele.pass-lanneau, pascale.bendotti, luca.brunod-indrigo\}@edf.fr}
}

\begin{document}

\maketitle

\begin{abstract}
The concept of anchored solutions is proposed as a new robust optimization approach to the Resource-Constrained Project Scheduling Problem (RCPSP) under processing times uncertainty. The Anchor-Robust RCPSP is defined, to compute a baseline schedule with bounded makespan, sequencing decisions, and a max-size subset of jobs with guaranteed starting times, called anchored set. It is shown that the Adjustable-Robust RCPSP from the literature fits within the framework of anchored solutions. The Anchor-Robust RCPSP and the Adjustable-Robust RCPSP can benefit from each other to find both a worst-case makespan and a baseline schedule with an anchored set.
A dedicated graph model for anchored solutions is proposed for budgeted uncertainty. Compact MIP reformulations are derived for both the Adjustable-Robust RCPSP and the Anchor-Robust RCPSP. Dedicated heuristics are designed based on the graph model. For both problems, the efficiency of the proposed MIP reformulations and heuristic approaches is assessed through numerical experiments on benchmark instances.\\

\end{abstract}

\section{Introduction}

The Resource-Constrained Project Scheduling Problem (RCPSP) is to find a min makespan schedule for the jobs of a project under precedence constraints and scarce resource availability. This problem finds a large variety of applications in the industry, see, e.g., \citep{ArtiguesDemasseyNeronRCPSP} for a survey. The RCPSP is a classical problem that has received a lot of attention from the scheduling community. It is known to be not only NP-hard but also computationally challenging to solve. Methods for solving the RCPSP include heuristic algorithms, or exact methods based, e.g., on constraint programming or integer programming.

The RCPSP is ubiquitous in applications, where it often appears with uncertain parameters, as it can be expected from any real-life problem.
Project scheduling under uncertainty has been studied with a variety of criteria and uncertainty modelling choices, see \citep{HerroelenLeusProjectSchedulingUncertainty} or \citep{HazirUlusoy2020} for surveys. In the present work, we focus on uncertainty in the processing times of the jobs. In the spirit of robust optimization, it is assumed that processing times of jobs lie in a given uncertainty set.

As pointed out in \citep{HazirUlusoy2020}, research on robust optimization for the RCPSP is rather scarce, compared to stochastic project scheduling.
In \citep{ArtiguesLeusTallaNobibon2013} the authors studied a robust problem under a min regret criterion.
In \citep{Bruni2016} the authors introduced the Adjustable-Robust RCPSP, derived from the general concept of adjustable robustness from \citep{BGGNAdjustable}.
Sequencing decisions (corresponding to resource allocation decisions) are fixed before uncertainty is known; by contrast the starting times of jobs adjust to the uncertainty realization.
The goal is to minimize the worst-case makespan.
This problem was solved using Benders-like decomposition approaches in \citep{Bruni2016,Bruni2018}.
While writing up this article, we found out about \citep{BoldGoerigk2020} where the authors studied the Adjustable-Robust RCPSP with similar results to ours for this problem.

Adjustable-robustness outputs a worst-case makespan, but it gives no schedule to decide in advance, since starting times depend on the uncertainty realization. However, finding a precomputed \emph{baseline schedule} is often required in applications to prepare the implementation of the schedule. 
For project scheduling under precedence constraints only, the computation of a baseline schedule was tackled in \citep{HerroelenStablePreschedule}. The baseline schedule can be optimized for a stability criterion, so that if the schedule must change after uncertainty realization, decisions are not changed too much w.r.t. baseline decisions. In this vein the criterion of \emph{anchored jobs} was introduced in \citep{CPMEJOR} and \citep{AnchRobPSPHal}. Anchored jobs are defined as jobs whose starting times are guaranteed w.r.t. an uncertainty set. The anchor-robust approach is to look for a baseline schedule with bounded makespan and a maximum number of anchored jobs.
\medbreak

\emph{Contributions.}
This article presents a new framework and solution methods in robust optimization for the RCPSP.
Anchored solutions are defined as a new concept in two-stage robust optimization for the RCPSP.
The Anchor-Robust RCPSP is defined, to find a baseline schedule with bounded makespan, sequencing decisions, and a maximum number of anchored jobs.
The connection with Adjustable-Robust RCPSP is pointed out. We show that these two problems can benefit from each other, in order to find both a worst-case makespan and a baseline schedule.
We extend algorithmic tools from the literature and obtain a graph model and compact MIP reformulations for adjustable-robustness and anchor-robustness in a unified way.
Heuristic algorithms are also proposed based on the same graph model. We investigate numerical performance of the proposed compact reformulation for the Adjustable-Robust RCPSP on instances based on the PSPLib. We point out that the overhead computational effort for solving the Adjustable-Robust RCPSP, in comparison with the deterministic RCPSP, is low. We follow on by investigating the efficiency of the proposed heuristic approaches to solve larger instances. We then provide numerical results for the Anchor-Robust RCPSP, for both MIP reformulation and heuristics. We highlight the good properties of anchor-robust solutions for finding a trade-off between makespan and guarantee of starting times.

\medbreak

\emph{Outline.}
Section~\ref{sec:prelim} is devoted to preliminaries.
In Section~\ref{sec:models}, anchored solutions are defined, together with the Anchor-Robust RCPSP.
In Section~\ref{sec:graphMIP}, graph models are presented. Compact reformulations and heuristics for the Adjustable-Robust RCPSP and for the Anchor-Robust RCPSP are proposed.
Sections~\ref{sec:num:exactAdj} to \ref{sec:num:heurAnch} are devoted to computational results.
In Section~\ref{sec:num:exactAdj}, numerical findings for the MIP for Adjustable-Robust RCPSP are reported.
In Section~\ref{sec:num:heurAdj}, heuristics for the Adjustable-Robust RCPSP are numerically assessed.
In Section~\ref{sec:num:exactAnch}, numerical findings for the MIP for Anchor-Robust RCPSP are reported.
In Section~\ref{sec:num:heurAnch}, numerical results of a heuristic for the Anchor-Robust RCPSP are reported.

\section{Preliminaries}
\label{sec:prelim}

\emph{Deterministic RCPSP.}
Let us define formally the Resource-Constrained Project Scheduling Problem (RCPSP).
The set of jobs is $\Jobs$, and $\AllJobs = \Jobs \cup \{s,t\}$ with $s$ and $t$ dummy jobs representing the beginning and the end of the project. The precedence graph $(\AllJobs, \Arcset) $ is directed and acyclic, with source $s$ and sink $t$.
Jobs have processing times $p \in \RealJp$. Let $(\AllJobs, \Arcset, p)$ denote the arc-weighted version where arc $(i,j) \in \Arcset$ has length $p_i$, and $p_s = 0$.
A schedule of precedence graph $(\AllJobs, \Arcset, p)$ is a vector $x \in \Real^{\AllJobs}_+$ satisfying precedence constraints $x_j - x_i \geq p_i$ for every arc $(i,j) \in \Arcset$.
The problem of finding a schedule of $(\AllJobs, \Arcset, p)$ of minimum \emph{makespan} $x_t$%$= \max_{i \in J} x_i+p_i$
, is referred to as PERT scheduling or project scheduling under (only) precedence constraints. This problem is solvable in polynomial time, the optimal value being the length of a longest \stpath\ path in the precedence graph.

Let $\Ress$ denote a set of renewable resources. Each job $i \in \Jobs$ has a resource requirement $\rik \geq 0$ of resource $k \in \Ress$. Let $\Rk \geq 0$ be the resource availability of resource $k$.
Given schedule $x \in \Real^{\AllJobs}$ and date $d \geq 0$, let $A_x(d) = \{ i \in J\colon\ x_i \leq d < x_i + p_i \}$ be the set of jobs under execution at date $d$ in schedule $x$.
Schedule $x$ satisfies resource constraints associated with requirements $(\rik)_{i \in J, k \in \Ress}$ and availabilities $(\Rk)_{k \in \Ress}$ if $\sum_{i \in A_x(x_j)} \rik \leq \Rk$ for every $k \in \Ress$ for every $j \in \AllJobs$. That is, at every date $x_j$ where a job $j$ starts, the total quantity of resource used by jobs being executed is at most the resource availability.
Given instance $\Inst = ((\AllJobs, \Arcset, p), (\rik)_{i \in J, k \in \Ress}$, $(\Rk)_{k \in \Ress})$,
the RCPSP is to find a schedule $x \in \Real^{\AllJobs}_+$ satisfying
precedence constraints of $(\AllJobs, \Arcset, p)$ and resource constraints,
so that the makespan $x_t$ is minimized.

\medbreak
\emph{Uncertainty sets.}
Consider processing times uncertainty. Processing times have nominal values $p \in \RealJp$ and their real value may be some $p+\delta$, where deviation $\delta$ is assumed to be lying in an uncertainty set $\Delta \subseteq \Real^J_+$.
In the present work we consider budgeted uncertainty, which is broadly used in robust optimization since its introduction in \citep{PriceBS}. The budgeted uncertainty set associated with vector $\dhat \in \RealJp$ and integer $\Gamma$ is defined by $\Delta = \{ \delta= (\dhat_i u_i)_{i \in \Jobs}\colon\ u \in \BinJobs,\ \sum_{i \in \Jobs} u_i \leq \Gamma \}$. 
Vector $\dhat$ corresponds to worst-case deviations of each processing time. Parameter $\Gamma$ is the \emph{uncertainty budget}, representing that at most $\Gamma$ processing times may deviate from their nominal values at a time.
If $\Gamma=0$, set $\Delta$ is reduced to the singleton $\{0\}$, and robust approaches reduce to the deterministic RCPSP.
Note that $\Delta$ is defined as a discrete set, and note a polytope; a result from \citep{BGGNAdjustable} applicable to all robust approaches we consider, is that $\Delta$ can be convexified without changing the solutions.
If $\Gamma= |\Jobs|$, the convex hull of set $\Delta$ is the \emph{box uncertainty set} $\Pi_{i \in \Jobs} [0, \dhat_i]$, also known as interval uncertainty set.

\medbreak
\emph{Static-robust approach.}
Let $\Inst =((\AllJobs, \Arcset,p), (\rik), (\Rk))$ denote the nominal RCPSP instance.
Let also $\InstDelta = ((\AllJobs, \Arcset,\pplusd), (\rik), (\Rk))$ be the instance where processing times are affected by deviation $\dinD$.
Following the work of \citep{Soyster} a first attempt is to define the Static-Robust RCPSP:\\
Given nominal instance $\Inst$ and set $\Delta$, the problem is to
find a schedule $x^+$ that is feasible for $\InstDelta$ for every $\dinD$, so as to minimize makespan $x^+_t$.\\
Schedule $x^+$ is said to be a static-robust solution.
In general, the static-robust approach is known to be overly conservative, in the sense that the cost of a static-robust solution is much more than the cost of a solution to the deterministic problem. 
This increase in cost is called the \emph{price of robustness} \citep{PriceBS}.
Budgeted uncertainty sets were introduced with the purpose of reducing the conservatism of the robust problem: the smaller $\Gamma$ the smaller the price of robustness.
However, we point out that for project scheduling, the uncertainty budget has no impact on the price of robustness.

\begin{prop}
\label{prop:static}
For every budget $\Gamma \in \{1, \dots, |J|\}$, solutions of the Static-Robust RCPSP for budgeted uncertainty set $\Delta = \{ (\dhat_i u_i)_{i \in \Jobs}\colon\ u \in \BinJobs,\ \sum_{i \in \Jobs} u_i \leq \Gamma \}$ are the same as solutions of the Static-Robust RCPSP for box uncertainty set $\Pi_{i \in J} [0, \dhat_i]$.
\end{prop}

\begin{proof}
Let $\Gamma < |J|$. Clearly a static-robust solution for box uncertainty is also static-robust for $\Delta$ since $\Delta \subseteq \Pi_{i \in J}[0, \dhat_i]$. Conversely, assume that $x$ is feasible for uncertainty set $\Delta$. 
Let $(i,j) \in \Arcset$. It holds that $x_j - x_i \geq p_i + \delta_i$ for every $\delta \in \Delta$. In particular for every budget $\Gamma \geq 1$, the uncertainty realization where $\delta_i = \dhat_i$ is in $\Delta$. Hence $x_j - x_i \geq p_i + \dhat_i$, and $x$ is a schedule of $(\AllJobs, \Arcset, p+\dhat)$. Since $\dhat$ is an upper-bound on the box uncertainty set, schedule $x$ is a static-robust solution for box uncertainty.
\end{proof}
This observation motivates the study of other robust approaches, where the schedule could be adapted to uncertainty realization.

\section{Anchor-Robust approach for the RCPSP}
\label{sec:models}

In this section we propose the concept of anchor-robustness for the RCPSP.\\
We first recall the flow formulation for the RCPSP, which serves as a basis for the anchor-robust approach.
Anchored sets are then formally defined.
The Anchor-Robust RCPSP is defined.
Finally we highlight the connection with the Adjustable-Robust RCPSP from the literature.

\subsection{Flow formulation for the RCPSP}

Let us first present the flow formulation for the RCPSP from the literature \citep{Artigues2003}. Let $\Pairs = \{ (i,j), i,j \in \AllJobs, i \not= j \}$. With a schedule $x$ satisfying resource constraints, one can associate a so-called \emph{resource flow} $f$. The resource flow is a collection $(f^k)_{k \in \Ress}$ of \stpath flows in the graph $(\AllJobs, \Pairs)$. Flow $f^k$ has value $\Rk$, and the quantity flowing through every job $i \in J$ equals the resource requirement $\rik$. Intuitively, the flow $f^k$ represents how units of resource $k$ are passed from a job to another.
If $\fijk > 0$ for some $k \in \Ress$, i.e., if job $j$ uses some units of resource $k$ after completion of job $i$, then an extra precedence constraint $x_j - x_i \geq p_i$ must be satisfied by schedule $x$. To represent such extra precedence constraints, consider vector $\sigma \in \BinPairs$.
Let $\ArcsetSigma = \{ (i,j) \in \Pairs\colon\ \sigmaij = 1\}$.
Vector $\sigma \in \BinPairs$ is a \emph{sequencing decision} (also called sufficient selection in the literature) associated with resource flow $f$ if:
(i) $\sigma$ contains all original precedence constraints, i.e., $\Arcset \subseteq \ArcsetSigma$;
(ii) $\sigma$ contains all arcs corresponding to non-zero flow values, i.e., $\{ (i,j) \in \Pairs\colon\ \exists\ k \in \Ress,\ \fijk > 0\} \subseteq \ArcsetSigma$.
Using the resource flow, the following formulation for the RCPSP was proposed in \citep{Artigues2003}:
\begin{center}
\begin{tabular}{l l l l}
$\min$ & $x_t$ \\
%  s.t. &  \\
s.t. & $x_j - x_i \geq p_i - \Mbar(1-\sigma_{ij})$ & $\forall i,j \in \Pairs$& (1)\\
& $\sigmaij = 1$ & $\forall (i,j) \in \Arcset$ & (2)\\
& $\fijk \leq \min\{\rik,\rjk\} \sigmaij$ & $\forall i,j \in \Pairs$, $k \in \Ress$ & (3)\\
& $\sum_{j: (i,j) \in \Pairs} \fijk = \tilde{r}_{ik}$ & $\forall k \in \Ress$, $\forall i \in \Jobs \cup \{s\}$ & (4)\\
& $\sum_{j: (j,i) \in \Pairs} \fjik = \tilde{r}_{ik}$ & $\forall k \in \Ress$, $\forall i \in \Jobs \cup \{t\}$ & (5)\\
& $\fijk \geq 0$ & $\forall i,j \in \Pairs, k \in \Ress$ &(6)\\
& $\sigma_{ij} \in \{0,1\}$ & $\forall i,j \in \Pairs$ & (7)\\
& $x_j \geq 0 $ & $\forall j \in \AllJobs$ & (8)\\
\end{tabular}
\end{center}
where $\tilde{r}_{ik} = \Rk$ if $i = s$ or $i= t$, and $\tilde{r}_{ik} = r_{ik}$ otherwise, and $\Mbar$ is an upper bound on the optimal value.
Constraint (1) imposes that vector $x \in \Real^{\AllJobs}_+$ is a schedule of $(\AllJobs, \ArcsetSigma,p)$, $\Mbar$ being a bigM value. Remark that for $\sigma \in \BinPairs$ and non-zero processing times, the existence of a schedule $x$ of $(\AllJobs, \ArcsetSigma, p)$ implies the absence of circuits in $\ArcsetSigma$.
Constraint (2) imposes that $\Arcset \subseteq \ArcsetSigma$. Constraints (4)--(5) define every $(\fijk)_{i,j \in \Pairs}$ to be a flow in the complete graph $(\AllJobs, \Pairs)$ and the quantity of resource $k$ going out of $s$ (resp. through $i \in \Jobs$) to be equal to $\Rk$ (resp. $r_{ik}$). Constraint (3) imposes that if a non-zero flow goes from $i$ to $j$ then there must be a precedence relation from $i$ to $j$ in $\sigma$.

Let $\Sigmas$ denote the set of all sequencing decisions, i.e., $$\Sigmas = \{\sigma \in \BinPairs\colon\ \exists f \text{ such that } (\sigma, f) \text{ satisfies }(2)\text{--}(7)\}.$$
Importantly, flow $f$ does not interfere directly with schedule $x$ but only implies that $\sigma$ as an element of $\Sigmas$. Hence a solution of the RCPSP can be represented as a pair $(x, \sigma)$ with $\sigma \in \Sigmas$ and $x$ a schedule of $(\AllJobs, \ArcsetSigma, p)$.

The flow formulation is known to have a poor linear relaxation bound, but it is compact and its size is independent from the time horizon \citep{Kone2013}. In the sequel, it will be shown that the flow formulation can be used as the underlying structure of robust formulations.

\subsection{Anchored sets}

Let us now propose the concept of anchor-robustness for the RCPSP, based on the definition of so-called anchored sets.
Anchor-robustness fits within the general framework of 2-stage robustness (or adjustable robustness) introduced in \citep{BGGNAdjustable}. In this framework,
first-stage decisions are made in a first stage before uncertainty is known, second-stage decisions are made after the uncertainty realization is observed in a second (or recourse) stage.
As second-stage decisions depend on the uncertainty realization, they are also called adjustable. By contrast, there are no second-stage decisions in static robustness. The existence of adjustable decisions adds flexibility to the robust approach, and the price of robustness is decreased \citep{BGGNAdjustable}.

We consider the following approach for 2-stage decisions in project scheduling. A baseline solution $(x, \sigma)$ is chosen in first stage, feasible for the instance $\Inst$. Then, the real instance may be $\InstDelta$: the schedule can be changed in second stage into a new schedule $\xd$ of $(\AllJobs, \ArcsetSigma, \pplusd)$, with sequencing decision $\sigma$ unchanged.
The decision maker may also want to guarantee starting times, so that the starting times of some jobs are the same in baseline schedule $x$ and second-stage schedule $\xd$.
To formalize this idea, let us introduce anchored jobs, following the line opened in \citep{AnchRobPSPHal} for PERT scheduling.

\begin{definition}
\label{def:anchoredSet}
Let $(z, \sigma)$ be a solution of the RCPSP instance $\Inst$. Let $\Hbar \subseteq \AllJobs$.\\
The set $\Hbar$ is \emph{anchored} w.r.t. schedule $z$ and sequencing decision $\sigma$ if for every $\dinD$, there exists a schedule $z^{\delta}$ of $(\AllJobs, \ArcsetSigma, \pplusd)$ such that $z_i = z^{\delta}_i$ for every $i \in \Hbar$.
\end{definition}

We say equivalently that $\Hbar$ is an anchored set or a subset of anchored jobs. An anchored set $\Hbar$ corresponds to jobs whose starting times are guaranteed in the baseline schedule $z$: indeed for every realization in the uncertainty set, it is possible to repair the baseline schedule $z$ into a feasible schedule $z^{\delta}$ of the new instance $(\AllJobs, \ArcsetSigma, \pplusd)$ without changing the starting times of anchored jobs.
\medbreak
In the sequel, we will consider \emph{anchored solutions} given as triplets $(z, \sigma, \Hbar)$, formed with a solution $(z,\sigma)$ of $\Inst$ with baseline schedule $z$ and sequencing decision $\sigma$, and subset of jobs $\Hbar$ anchored w.r.t. $z$ and $\sigma$.
\medbreak
Static-Robust solutions correspond exactly to solutions with anchored set $\Hbar = \AllJobs$. Indeed if $\AllJobs$ is anchored w.r.t. $z$ and $\sigma$, then $z$ is a schedule of $\PrecGraphAsigmaPD$ for every $\dinD$. Then $(z,\sigma)$ is a solution of $\InstDelta$ for every $\dinD$: it is a static-robust solution.

\subsection{The Anchor-Robust RCPSP}

Let us now define the optimization problem Anchor-Robust RCPSP:\\
Given RCPSP instance $\Inst = ((\AllJobs, \Arcset, p), (\rik)_{i \in J, k \in \Ress}, (\Rk)_{k \in \Ress})$ and a deadline $M \geq 0$, find a sequencing decision $\sigma \in \Sigmas$, a baseline schedule $z$ of $(\AllJobs, \ArcsetSigma, p)$ with makespan at most $M$, and a subset of jobs $H \subseteq J$, such that $H \cupt$ is anchored w.r.t. $z$ and $\sigma$, and the number of anchored jobs $|H|$ is maximized.
\medbreak
The Anchor-Robust RCPSP can be seen as a robust 2-stage problem, and written as the following mathematical program:
\begin{center}
\begin{tabular}{l l l}
    $\max$        & $\min$ & $\max \quad |H|$ \\
    $\sigma \in \Sigmas$ & $\dinD$ &  $z^{\delta}$ schedule of $\PrecGraphAsigmaPD$\\
    $z$ schedule of $\PrecGraphAsigmaP$ & & $\zd_i = z_i$ $\forall i \in H$ \\
    $z_t \leq M$ && $\zd_t = z_t$\\
    $H \subseteq J$ \\
\end{tabular}
\end{center}
The inner min/max problem has finite value $|H|$ if the set $H \cupt$ is anchored w.r.t. $z$ and $\sigma$; otherwise it has infinite value.
The Anchor-Robust RCPSP is thus to find an anchored solution $(z, \sigma, H\cupt)$ where the makespan of $z$ is bounded by deadline $M$, and the size of $H$ is maximized.
Note that the final job $t$ is forced to be in the anchored set $H \cupt$. This offers a guarantee on the worst-case makespan, in the sense that $\zd_t \leq M$ for every $\dinD$. \medbreak
The connection with static-robustness is as follows. There exists a solution with anchored set $\Hbar = \AllJobs$ if and only if $M \geq M_{stat}$, where $M_{stat}$ is the optimal value of the Static-Robust RCPSP. If $M < M_{stat}$ a solution of the Anchor-Robust RCPSP only has a subset of the jobs that are anchored, but of maximum size.

\medbreak
For project scheduling with precedence constraints only, the Anchor-Robust Project Scheduling Problem was introduced in \citep{AnchRobPSPHal}. In the Anchor-Robust RCPSP, the schedule in second stage is subject to precedence constraints only since the sequencing decision $\sigma \in \Sigmas$ is fixed. Thus tools from \citep{AnchRobPSPHal} can be used.

\medbreak
The Anchor-Robust RCPSP provides a baseline solution where decisions are guaranteed in two ways. First, some starting times, i.e., starting times of jobs in $H$, are not modified in second stage by definition of anchored sets. Second, the sequencing decision $\sigma$ is a first-stage decision and it is not to be modified in second stage.
We emphasize that fixing sequencing decision in advance can be of practical interest. In applications, the sequencing decision may correspond, e.g., to the sequence of operations performed by workers with specific skills or equipment. They would prefer not to revise the order in which operations are performed; by contrast, it may be acceptable to adapt some starting times.
From a computational point of view, sequencing decisions in the second stage would lead to very difficult optimization problems. Namely, if the sequencing decision is in second stage, even deciding if the starting time of a job can be guaranteed, is NP-hard.
More details on the complexity of rescheduling problems was discussed in \citep{ISCOAnchResched}.

\subsection{Connection with the Adjustable-Robust RCPSP}

An adjustable-robust approach for the RCPSP was proposed in \citep{Bruni2016}. First stage decision is the sequencing decision $\sigma \in \Sigmas$; schedule $x$ is decided in second stage. The Adjustable-Robust RCPSP writes as the following mathematical program:
\begin{center}
\begin{tabular}{l l  r l}
$\min$  & $\max$ & $\min$ & $\xd_t$\\ 
$\sigma \in \Sigmas$ & $\dinD$ & s.t.& $\xd$ schedule of $\PrecGraphAsigmaPD$ \\
\end{tabular}
\end{center}
Let $\QGamma(\sigma)$ denote the inner max-min problem, i.e., the worst-case makespan for sequencing decision $\sigma \in \Sigmas$. The Adjustable-Robust RCPSP is then to minimize $\QGamma(\sigma)$ for $\sigma \in \Sigmas$.
\medbreak
Contrary to the anchor-robust approach, the Adjustable-Robust RCPSP does not include the computation of a baseline schedule. However the Adjustable-Robust RCPSP is related to anchored sets by the following observation. 
\begin{observation}
\label{obs}
Let $\sigma \in \Sigmas$. The worst-case makespan $\QGamma(\sigma)$ is equal to the minimum makespan of $z$, where $(z,\sigma,\{t\})$ is an anchored solution.
\end{observation}
\begin{proof}
By definition $\QGamma(\sigma)$ is the minimum value $M$ such that for every $\dinD$, there exists a schedule $\xd$ of $\PrecGraphAsigmaPD$ with makespan at most $M$. 
Equivalently, it is the minimum $M$ such that $z$ is a schedule of $\PrecGraphAsigmaP$ with $z_t=M$, and $\sinkt$ is anchored w.r.t. $z$ and $\sigma$, by definition of anchored solutions.
\end{proof}
\medbreak
A consequence of Observation~\ref{obs} is that there exists a solution of the Anchor-Robust RCPSP if and only if $M \geq M_{adj}$, where $M_{adj}$ is the optimal value of the Adjustable-Robust RCPSP. The singleton $\sinkt$ can be anchored if and only if $M \geq M_{adj}$.
Note that even with deadline $M = M_{adj}$, there exists solutions with anchored set $H \cupt$, $H \not= \varnothing$, as illustrated in Section~\ref{sec:num:exactAnch}. That is, it is possible to guarantee a non-trivial subset of starting times without increasing the worst-case makespan $M_{adj}$.
\medbreak
The Adjustable-Robust RCPSP and the Anchor-Robust RCPSP can benefit from each other in the following way. First the Adjustable-Robust RCPSP can be solved to determine a robust worst-case makespan $M_{adj}$. Then the Anchor-Robust RCPSP can be solved with deadline $M_{adj}$ to determine schedule $z^*$ satisfying this deadline, sequencing decision $\sigma^*$, along with an anchored set $H^*$.\\

\medbreak
Let us review the approaches proposed in the literature to solve the Adjustable-Robust RCPSP.
In \citep{Bruni2016} the problem was introduced and a Benders decomposition was proposed. The computation of $\QGamma(\sigma)$ for fixed $\sigma$ appears as subproblem. The authors then focused on budgeted uncertainty and proved that $\QGamma(\sigma)$ can be computed in polynomial time by dynamic programming in that case. Note that it is the same problem solved in \citep{Minoux2007}. The authors proposed enhancements to the Benders decomposition algorithm, evaluated on instances built upon the PSPLib.
In \citep{Bruni2018} the same authors investigated two new methods: a primal and a dual one. The dual method has the same subproblem as in \citep{Bruni2016} but other cuts associated with dual information are added on the fly. The primal method is a column-and-constraint generation scheme inspired from \citep{ZZ}. Computational comparison of the two decomposition approaches of \citep{Bruni2018} and the one of \citep{Bruni2016} is reported.
In \citep{BoldGoerigk2020} a compact reformulation for the Adjustable-Robust RCPSP is proposed, independently from the present work.

\section{Graph model and compact MIP reformulations}
\label{sec:graphMIP}

In this section, a graph model is proposed to characterize anchored sets and solutions of the Anchor-Robust RCPSP. Compact MIP reformulations are deduced for the Anchor-Robust RCPSP and for the Adjustable-Robust RCPSP. Heuristics based on the graph model are also proposed.

\subsection{Layered graph}
\label{sec:layeredGraph}

Let us propose a graph model based on the \emph{layered graph} previously introduced in \citep{AnchRobPSPHal} for PERT scheduling.
Let $\Delta$ be the budgeted uncertainty set defined by deviation $\dhat \in \RealJp$ and budget $\Gamma$, and let nominal processing times be $p \in \RealJp$.
\medbreak
Let $(\AllJobs, \Arcset, p)$ be a precedence graph.
For $\Hbar \subseteq \AllJobs$, the \emph{layered graph} $\Glay(\Arcset, \Hbar)$ is defined as follows. It contains $\Gamma+1$ \emph{layers}, indexed from $0$ to $\Gamma$. Each layer $\gamma$ contains a copy of the set of jobs, denoted by $\gam{i}$, $i \in \AllJobs$. The arc-set of the layered graph contains three types of arcs. \emph{Horizontal arcs} are arcs $(\gam{i}, \gam{j})$ for every $\gamma$ and $(i,j) \in \Arcset$, with arc-weights $p_i$.
\emph{Transversal arcs} are arcs $(\gamp{i}, \gam{j})$ for every $\gamma < \Gamma$ and $(i,j) \in \Arcset$, with arc-weights $p_i + \dhat_i$. \emph{Vertical arcs} are arcs $(\gam{j}, \Gam{j})$ for every $j \in \Hbar$, $\gamma < \Gamma$, with arc-weights $0$.
\medbreak
Note that the layered graph has all information from $\Delta$, since the number of layers depends on the uncertainty budget $\Gamma$, and deviations $\dhat$ appear on arc-weights of transversal arcs.
Note also that if the graph $(\AllJobs, \Arcset)$ is acyclic, then $\Glay(\Arcset,\Hbar)$ is acyclic.
In \citep{AnchRobPSPHal} the set $\Hbar$ was defined as a subset of $J$ and not $\AllJobs$, but all results extend to the case of a set $\Hbar$ containing $s$ or $t$. Let us recall the main result we use from this previous work.

\begin{theorem}[\cite{AnchRobPSPHal}]
\label{thm:OR}
Let $z$ be a schedule of $(\AllJobs, \Arcset, p)$ and $\Hbar \subseteq \AllJobs$. The following assertions are equivalent:\\
(i) For every $\dinD$ there exists $\zd$ a schedule of $(\AllJobs, \Arcset, \pplusd)$ such that $\zd_i = z_i$ for every $i \in \Hbar$;\\
(ii) There exists $x$ a schedule of $\Glay(\Arcset, \Hbar)$ such that $\Gam{x}_i = z_i$ for every $i \in \AllJobs$.
\end{theorem}

\medbreak
Let us now derive a characterization of anchored sets for the RCPSP. It follows directly from Definition~\ref{def:anchoredSet} and Theorem~\ref{thm:OR} that

\begin{theorem}
\label{thm:caracAnch}
Let $(z,\sigma)$ be a solution of the RCPSP instance $\Inst$. Let $\Hbar \subseteq \AllJobs$.\\
The set $\Hbar$ is anchored w.r.t. sequencing decision $\sigma$ and schedule $z$ if and only if there exists $x$ a schedule of $\Glay(\ArcsetSigma, \Hbar)$ such that $\Gam{x}_i = z_i$ for every $i \in \AllJobs$.
\end{theorem}

Combining Observation~\ref{obs} and Theorem~\ref{thm:caracAnch}, it follows that

\begin{corollary}
\label{cor:layeredQrond}
For every $\sigma \in \Sigmas$, the worst-case makespan $\QGamma(\sigma)$ is equal to the minimum value of $\Gam{x}_t$ for $x$ a schedule of $\Glay(\ArcsetSigma, \{t\})$.
\end{corollary}

Hence the layered graph can be used to compute the worst-case makespan $\QGamma(\sigma)$, since the minimum value of $\Gam{x}_t$ for $x$ a schedule of $\Glay(\ArcsetSigma, \{t\})$ is equal to the length of the longest \odpath{\Gam{s}}{\Gam{t}}path in $\Glay(\ArcsetSigma, \{t\})$. Recall that $\Glay(\ArcsetSigma, \{t\})$ is acyclic when $(\AllJobs, \ArcsetSigma)$ is acyclic. Such longest path can be computed in polynomial time by dynamic programming.

\subsection{Compact MIP reformulations}
\label{sec:reformulation}

Let us now introduce new MIP reformulation for the Anchor-Robust and Adjustable-Robust RCPSP.
The formulations involve the following decision variables:\\
-- variables $f$ and $\sigma$ as in the flow formulation for the RCPSP;\\
-- continuous variables $\gam{x}_j \geq 0$ for every $j \in \AllJobs$, $\gamma \in \{0, \dots, \Gamma \}$;\\
-- for the Anchor-Robust RCPSP, binary variables $h \in \BinJobs$.

\medbreak

\begin{theorem}
\label{thm:MIPAnch}
The Anchor-Robust RCPSP under budgeted uncertainty admits the following compact MIP reformulation $(\FormAnch)$.
\end{theorem}

\begin{tabular}{l l l l c}
$(\FormAnch)$ \hspace{0.5cm} & $\max$ & $\sum_{i \in J} h_i$ \\
& s.t. & $\gam{x}_j - \gam{x}_i \geq p_i - M(1-\sigma_{ij})$ & $\forall i,j \in \Pairs$, $\forall \gamma \leq \Gamma$ & (a)\\
& & $\gam{x}_j - \gamp{x}_i \geq p_i +\dhat_i - M(1-\sigma_{ij})$ & $\forall i,j \in \Pairs$, $\forall \gamma < \Gamma$ & (b)\\
& & $\Gam{x}_t - \gam{x}_t \geq 0$ & $\forall \gamma < \Gamma$ & (c)\\
& & $\Gam{x}_j - \gam{x}_j \geq -M( 1- h_j)$ & $\forall j \in J$, $\forall \gamma < \Gamma$ & (d) \\
& & $\Gam{x}_t \leq M$ & & (e)\\
& & $\sigma$, $f$ satisfy (2)--(5) \\
& & $\fijk \geq 0$ & $\forall i,j \in \Pairs, k \in \Ress$ & (6)\\
& & $\sigma_{ij} \in \{0,1\}$ & $\forall i,j \in \Pairs$ & (7)\\
& & $\gam{x}_j \geq 0 $ & $\forall j \in \AllJobs$, $\forall \gamma \leq \Gamma$ & (8)\\
& & $h_j \in \{0,1\}$ & $\forall j \in J$ & (9) \\
\end{tabular}

\begin{proof}
Constraints (2)--(7) ensure that $\sigma \in \Sigmas$.
By Theorem~\ref{thm:caracAnch}, the Anchor-Robust RCPSP is to find a set $H$, a sequencing decision $\sigma$, and a schedule $x$ of $\Glay(\ArcsetSigma,H\cupt)$ such that $\Gam{x}_t \leq M$. Indeed if $x$ is a schedule of $\Glay(\ArcsetSigma,H\cupt)$, then $\Gam{x}$ is a schedule of $\PrecGraphAsigmaP$ due to horizontal arcs in layer $\Gamma$: hence $\Gam{x}$ will be the baseline schedule. Let $h \in \BinJobs$ be the incidence vector of set $H$. It remains to show that (a)--(d) correctly enforce constraints from $\Glay(\ArcsetSigma, H \cupt)$.
For $\sigmaij = 1$, constraints (a) (resp. (b)) enforce horizontal arcs (resp. transversal arcs) constraints. Constraints (c) enforce vertical arc constraints between copies of $t$. For $h_j=1$, constraints (d) enforce vertical arc constraints between copies of $j$.
Hence it suffices to check that for $\sigmaij=0$ (resp. $h_j=0$) constraints (a)--(b) (resp. constraints (d)) are valid for any schedule $x$.
Note first that the deadline constraint (e) and vertical arc constraints (c) imply $\gam{x}_t \leq M$ for every $\gamma \leq \Gamma$. Thus $\gamp{x}_i + p_i +\dhat_i \leq M$ for every $\gamma < \Gamma$, and $\gam{x}_i + p_i \leq M$ for every $\gamma \leq \Gamma$.
Hence $\gam{x}_i+p_i - \gam{x}_j \leq \gam{x}_i+p_i \leq M$, and constraint (a) is valid if $\sigmaij=0$.
Similarly $\gamp{x}_i+p_i+\dhat_i - \gam{x}_j \leq \gamp{x}_i+p_i+\dhat_i \leq M$, and constraint (b) is valid if $\sigmaij=0$.
Finally $\gam{x}_j - \Gam{x}_j \leq M$, and constraint (d) is valid if $h_j=0$.
\end{proof}

Note that deadline $M$ is used as a common bigM value in constraints (a), (b) and (d).

A similar result holds for the Adjustable-Robust RCPSP.

\begin{theorem}
\label{thm:MIPAdj}
The Adjustable-Robust RCPSP under budgeted uncertainty admits the following compact MIP reformulation $(\FormAdj)$, where $\Mbar$ is an upper-bound on the optimal value.
\end{theorem}

\begin{tabular}{l l l l c}
($\FormAdj$) \hspace{0.5cm} & $\min$ & $\Gam{x}_t$ \\
%  s.t. &  \\
& s.t. & $\gam{x}_j - \gam{x}_i \geq p_i - \barM(1-\sigma_{ij})$ & $\forall i,j \in \Pairs$, $\forall \gamma \leq \Gamma$ & (i)\\
& & $\gam{x}_j - \gamp{x}_i \geq p_i +\dhat_i - \barM(1-\sigma_{ij})$ & $\forall i,j \in \Pairs$, $\forall \gamma < \Gamma$ & (ii)\\
& & $\Gam{x}_t - \gam{x}_t \geq 0$ & $\forall \gamma < \Gamma$ & (iii)\\
& & $\sigma$, $f$ satisfy (2)-(5) \\
& & $\fijk \geq 0$ & $\forall i,j \in \Pairs, k \in \Ress$ &(6)\\
& & $\sigma_{ij} \in \{0,1\}$ & $\forall i,j \in \Pairs$ & (7)\\
& & $\gam{x}_j \geq 0 $ & $\forall j \in \AllJobs$, $\forall \gamma \leq \Gamma$ & (8)\\
\end{tabular}

\begin{proof}
Constraints (2)--(7) model that $\sigma \in \Sigmas$.
By Corollary~\ref{cor:layeredQrond} the worst-case makespan $\QGamma(\sigma)$ is the minimum value of $\Gam{x}_t$ for $x$ a schedule of $\Glay(\ArcsetSigma, \{t\})$.
It suffices to show that (i)-(ii)-(iii) correctly enforce constraints from $\Glay(\ArcsetSigma, \{t\})$.
The proof is very similar as the one of Theorem~\ref{thm:MIPAnch}.
For $\sigmaij = 1$, constraints (i) (resp. (ii)) enforce horizontal arcs (resp. transversal arcs) constraints. Constraints (iii) enforce vertical arc constraints between copies of $t$. 
Hence it suffices to check that for $\sigmaij=0$ constraints (i)--(ii) are valid for any optimal schedule $x$.
The value $\Mbar$ is an upper bound hence $\Gam{x}_t \leq \Mbar$ can be assumed w.l.o.g. Vertical arc constraints (iii) then imply $\gam{x}_t \leq \Mbar$ for every $\gamma \leq \Gamma$. Thus $\gamp{x}_i + p_i +\dhat_i \leq \Mbar$ for every $\gamma < \Gamma$, and $\gam{x}_i + p_i \leq \Mbar$ for every $\gamma \leq \Gamma$.
Hence $\gam{x}_i+p_i - \gam{x}_j \leq \gam{x}_i+p_i \leq \Mbar$, and constraint (i) can be imposed w.l.o.g. if $\sigmaij=0$.
Similarly $\gamp{x}_i+p_i+\dhat_i - \gam{x}_j \leq \gamp{x}_i+p_i+\dhat_i \leq \Mbar$, and constraint (b) can be imposed w.l.o.g. if $\sigmaij=0$.
\end{proof}

Note that $\Mbar$ can be set to $\Mbar = \sum_{i \in J} p_i+\dhat_i$.
\medbreak
The same compact reformulation for the Adjustable-Robust RCPSP was obtained independently in \citep{BoldGoerigk2020}.
\medbreak
We emphasize that obtaining such compact reformulations is rather an exception than a general rule for robust 2-stage problems. 
Hence the Anchor-Robust RCPSP and the Adjustable-Robust RCPSP can be solved directly by MIP solving. Besides this formulation, the literature for the Adjustable-Robust RCPSP contains only decomposition methods based on exponential formulations.
Even if such reformulations have reasonable (polynomial) size, they become difficult to solve for off-the-shelf MIP solvers when the size of the instance increases. This motivates the design of heuristic algorithms.

\subsection{Heuristic algorithms for the Adjustable-Robust RCPSP}
\label{sec:priorityRules}

Let us propose a framework for designing heuristics for the Adjustable-Robust RCPSP.
Consider a heuristic algorithm $\heurAlgo$ to solve the (deterministic) RCPSP; it is assumed that when applied to instance $\Inst$, algorithm $\heurAlgo$ outputs a sequencing decision $\heurAlgo(\Inst) = \sigmaHeur \in \Sigmas$. The heuristic algorithm $\heurAlgo$ can be used to design the following heuristic $\AdjGammaHeur$ for the Adjustable-Robust RCPSP:

\begin{algorithm}[H]
\KwData{Instance $\Inst$, set $\Delta$ with budget $\Gamma$, algorithm $\heurAlgo$}
\textbf{Let} $\sigmaHeur := \heurAlgo(\Inst)$\;
\textbf{Let} $\QGamma(\sigmaHeur)$ := longest \odpath{\Gam{s}}{\Gam{t}}path in $\Glay(\ArcsetSigmaHeur,\{t\})$\;
\Return solution $\sigmaHeur$ with value $\QGamma(\sigmaHeur)$ \;
\caption{Heuristic $\AdjGammaHeur$ for the Adjustable-Robust RCPSP}
\end{algorithm}

Algorithm $\AdjGammaHeur$ has complexity $C_{\heurAlgo}+O((|J|+|\Arcset|)\Gamma)$ where $C_{\heurAlgo}$ is the complexity of algorithm $\heurAlgo$. Indeed the computation of $\QGamma(\sigmaHeur)$ can be done by dynamic programming since $\Glay(\ArcsetSigmaHeur, \{t\})$ is acyclic. The vertex-set (resp. arc-set) of $\Glay(\ArcsetSigmaHeur, \{t\})$ has size $O(|J|\Gamma)$ (resp. $O(|\Arcset|\Gamma)$) leading to longest path computation complexity $O((|J|+|\Arcset|)\Gamma)$.

\medbreak

Let us now present a class of heuristic algorithms for the RCPSP to instantiate algorithm $\heurAlgo$. They are based on Parallel Schedule Generation Scheme (Parallel SGS) and priority rules. We refer to \citep{KolischHartmannRCPSPHeur} for details and references on heuristic algorithms for the RCPSP.
In Parallel SGS, a feasible schedule is built incrementally by time incrementation. At current time, the set of eligible jobs is formed with all jobs whose predecessors in $(\AllJobs, \Arcset, p)$ are scheduled and completed, and such that at current time there is enough resources available to start the job. A job is selected for the eligible set, and scheduled. If the eligible set is empty, time is incremented to the next date where a job completes.

To select a job from the eligible set, a common method is to use \emph{priority rules}. We consider static priority rules that are computed before Parallel SGS is executed. In that case, every job is given a priority $\pi(i)$. When a job is to be selected from the eligible set, the job with highest priority $\pi(i)$ is selected.
In the sequel, a total of 7 priority rules are considered, a trivial one and 6 priority rules from the literature:\\
  $\bullet$ Trivial rule (ID): $\pi(i)$ is the index of job $i$\\
  $\bullet$ Shortest Processing Time (SPT): $\pi(i) = -p_i$\\
  $\bullet$ Most Total Successors (MTS): $\pi(i)$ is the number of successors in the transitive closure of $(\AllJobs, \Arcset)$\\
  $\bullet$ Latest Finish Time (LFT): $\pi(i) = \bar{x}_i + p_i$ where $\bar{x}$ is the latest schedule of $(\AllJobs, \Arcset, p)$ with minimum makespan\\
  $\bullet$ Latest Starting Time (LST): $\pi(i) = \bar{x}_i$\\
  $\bullet$ Minimum Slack (MSLK): $\pi(i) = -(\bar{x}_i - \underline{x}_i)$, where $\underline{x}$ is the earliest schedule of $(\AllJobs, \Arcset, p)$\\
  $\bullet$ Greatest Rank Positional Weight (GRPW): $\pi(i) = p_i + \sum_{j: (i,j) \in \Arcset} p_j$\\

Importantly, ParallelSGS with priority rule $\pi$ can easily be executed so that it outputs a resource flow, and thus a sequencing decision $\sigmaHeur^{\pi} \in \Sigmas$. This was done, e.g., in \citep{Artigues2003}.
Let $\heurAlgo^{\pi}$ denote the algorithm corresponding to Parallel SGS with priority $\pi$.
Depending on the instance, it is not always the same priority rule that yields the best $\QGamma(\sigmaHeur^{\pi})$ value. Consider the following heuristic, denoted by BestRule:
\medbreak
\begin{algorithm}[H]
\KwData{Instance $\Inst$, set $\Delta$ with budget $\Gamma$}
\For{priority rule $\pi$ in \{ID, SPT, MTS, LFT, LST, MSLK, GRPW\}}{
    Let $\sigmaHeur^{\pi}$, $\QGamma(\sigmaHeur^{\pi})$ be the output of $\AdjGamma{\heurAlgo^{\pi}}$ \;
}
Let $\pi^* := \arg\min_{\pi}\QGamma(\sigmaHeur^{\pi})$ \;
\Return solution $\sigmaHeur^{\pi^*}$ with value $\QGamma(\sigmaHeur^{\pi^*})$\;
\caption{Heuristic BestRule for the Adjustable-Robust RCPSP}
\end{algorithm}
\medbreak
Note that the selected sequencing decision is the one that gives the best worst-case makespan value for the considered uncertainty set. Namely, the output of the BestRule heuristic depends on the uncertainty budget $\Gamma$.

\subsection{Heuristic for the Anchor-Robust RCPSP}

Due to the connection between Adjustable-Robust RCPSP and Anchor-Robust RCPSP, it can be expected that solving Anchor-Robust RCPSP will be computationally challenging, and heuristics will be needed to solve medium-size instances. Note that in the MIP reformulations, the formulation for Anchor-Robust RCPSP features the same variables as that of the Adjustable-Robust RCPSP, plus additional binary variables $h \in \BinJobs$.
Let us now propose a MIP-based heuristic for solving the Anchor-Robust RCPSP.
\medbreak
Consider the case where the deadline $M$ is computed by solving the Adjustable-Robust RCPSP in a first phase, i.e., $M = \QGamma(\sigmaHeur)$ for some sequencing decision $\sigmaHeur \in \Sigmas$. This can be done through the MIP formulation or a heuristic algorithm such as BestRule. 
\medbreak
Note first that the knowledge of $\sigmaHeur$ readily gives a feasible solution for the Anchor-Robust RCPSP with deadline $M =\QGamma(\sigmaHeur)$. Indeed with $x$ the earliest schedule of $\Glay(\ArcsetSigma, \sinkt)$, it holds that $(\Gam{x}, \sigmaHeur, \sinkt)$ is an anchored solution with baseline schedule $\Gam{x}$ respecting deadline $\Gam{x}_t \leq M$.
In general, the optimal solution of the Anchor-Robust RCPSP $(z^*, \sigma^*, H^* \cupt)$ can be with $\sigma^* \not= \sigmaHeur$.
Consider the following heuristic, which is to solve the Anchor-Robust RCPSP while enforcing $\sigma^* = \sigmaHeur$:
\medbreak
\begin{algorithm}[H]
\KwData{Instance $\Inst$, set $\Delta$ with budget $\Gamma$, sequencing decision $\sigmaHeur$, deadline $M= \QGamma(\sigmaHeur)$}
Let $(x^*, H^*) := \arg\max |H|$\\
\hspace{3.5cm} s.t. $H \subseteq J$, $x$ schedule of $\Glay(\ArcsetSigmaHeur, H)$ with $\Gam{x}_t \leq M$\;
\Return solution $(\Gam{x^*}, \sigmaHeur, H^*)$\;
\caption{Heuristic FixedSequence for the Anchor-Robust RCPSP}
\end{algorithm}

\newpage
The maximization step can be done by adapting the MIP from Theorem~\ref{thm:MIPAnch} to fixed sequencing decision $\sigmaHeur$. That is, $(x^*, H^*)$ is an optimal solution to the MIP formulation

\begin{center}
\begin{tabular}{l l l l}
$\max$ & $\sum_{i \in J} h_i$ \\
s.t. & $\gam{x}_j - \gam{x}_i \geq p_i$ & $\forall (i,j) \in \ArcsetSigmaHeur$, $\forall \gamma \leq \Gamma$ \\
& $\gam{x}_j - \gamp{x}_i \geq p_i +\dhat_i$ & $\forall (i,j) \in \ArcsetSigmaHeur$, $\forall \gamma < \Gamma$ & \\
& $\Gam{x}_t - \gam{x}_t \geq 0$ & $\forall \gamma < \Gamma$ & \\
& $\Gam{x}_j - \gam{x}_j \geq -M( 1- h_j)$ & $\forall j \in J$, $\forall \gamma < \Gamma$ & \\
& $\Gam{x}_t \leq M$ & & \\
& $\gam{x}_j \geq 0 $ & $\forall j \in \AllJobs$, $\forall \gamma \leq \Gamma$ & \\
& $h_j \in \{0,1\}$ & $\forall j \in J$ & \\
\end{tabular}
\end{center}

This is an instance of the Anchor-Robust Project Scheduling Problem for PERT scheduling from \citep{AnchRobPSPHal}. While the problem is still NP-hard, it is easier to solve than the Anchor-Robust RCPSP because there are no sequencing decision variables.
In the case where $\sigmaHeur$ is obtained by the BestRule heuristic, the FixedSequence heuristic is referred to as the BestRuleSequence heuristic.

\section{Computational results: MIP for Adjustable-Robust RCPSP}
\label{sec:num:exactAdj}

In this section numerical performances of the compact reformulation obtained in Section~\ref{sec:reformulation} for the Adjustable-Robust RCPSP are discussed. 

\subsection{Instances and settings}

The instances are built upon RCPSP instances from the PSPLib. The number of jobs is $n \in \{30, 60, 90, 120\}$. The instances with $n=30$ jobs are the same as \citep{Bruni2016,Bruni2018}. For fixed $n$, there are 480 RCPSP instances. There are 4 resources types ($|\Ress|=4$). The instances differ through three parameters:\\
-- the network complexity NC $\in \{1.5, 1.8, 2.1\}$ corresponding to the average degree of jobs in the precedence graph;\\
-- the resource factor RF $\in \{0.25, 0.50, 0.75, 1\}$ indicating the number of resources used by a job;\\
-- the resource strength RS $\in \{0.20, 0.50, 0.70, 1\}$ quantifying the size of resource conflicts.

Regarding uncertainty, deviation is defined by $\dhat = 0.5 p$ and the uncertainty budget is $\Gamma \in \{3,5,7\}$. This leads to 1440 instances for each value of $n$. The case of $\Gamma=0$, corresponding to deterministic RCPSP, will also be considered.

\medbreak
The MIP formulation $(\FormAdj)$ has been implemented using Julia 0.6.2, with JuMP 0.18.5. It is solved with CPLEX 12.8 on a PC under Windows 10 with Intel Core i7-7500U CPU 2.90GHz and 8 Go RAM. 
The upper bound $\barM$ is set to $\sum_{i \in \Jobs} (p_i+\dhat_i)$.
The time limit is set to 1200 seconds. % as in \citep{Bruni2016,Bruni2018}.

\medbreak
In the sequel we report averaged results. Detailed computational results can be found in the Appendix.

\subsection{Performance of $(\FormAdj)$ and impact of parameters on small instances}
\label{sec:res_adj_small}

Let us first consider instances with $n=30$ jobs, as these small instances were considered in the literature \citep{Bruni2016,Bruni2018}. In this section computational results are presented and the impact of parameters is studied to identify the hardest instances to solve with the considered formulation.

\subsubsection{Impact of the budget}

Let us first present computational results and analyze the impact of uncertainty budget $\Gamma$. In Table~\ref{tab:G357} the following results are reported for $\Gamma = 3,5,7$:\\
-- \#solved: the number of instances solved to optimality within time limit;\\
-- time: the computation time averaged on instances solved to optimality, in seconds;\\
-- \#unsolved: the number of instances not solved to optimality within time limit;\\
-- gap: the final gap averaged on instances not solved to optimality.

\begin{table}[H]
    \centering
    \footnotesize
    \begin{tabular}{c @{\hspace{2cm}} c c c c}
$\Gamma$ & \#solved & time(s) & \#unsolved & gap \\
\hline
3 & 345 & 39.53 & 135 & 19.26\% \\
5 & 344 & 44.65 & 136 & 19.01\% \\
7 & 337 & 50.98 & 143 & 19.17\% \\
\hline
all & 1026 & 45.01 & 414 & 19.15\% \\
\end{tabular}
    \caption{Performance of $(\FormAdj)$ depending on the budget $\Gamma$ for $n=30$ jobs.}
    \label{tab:G357}
\end{table}

Note first that direct implementation of the MIP reformulation allows us to solve 1026 instances over the total number of 1440 instances.
The impact of the uncertainty budget appears to be limited: when $\Gamma$ is increased, the number of solved instances and average computation time smoothly deteriorate.
\medbreak
The performance of $(\FormAdj)$ can be compared to state-of-the-art methods from the literature, that are decomposition methods of \citep{Bruni2016,Bruni2018}. The authors compare three methods; the Primal Method from \citep{Bruni2018} results to be the more efficient of the three. The Primal Method solves 767 instances out of 1440 within the time limit of 1200 seconds. The average time for solved instances is 113,13 seconds. The average final gap for unsolved instances is 13,48\%. % \str{restreindre aux sets non resolus par compact pour comparer?} 
Consequently, the MIP reformulation appears to be competitive with state-of-the-art decomposition methods.
We acknowledge that we did not re-implement and run the decomposition algorithms from \citep{Bruni2016,Bruni2018}. The comparison in terms of computation time of our compact reformulation and decomposition approaches is thus limited in significance. 
However the proposed approach is practically attractive since the implementation of a compact MIP formulation is much easier than the implementation of decomposition methods, where convergence issues may arise.

\subsubsection{Impact of PSPLib parameters}

Let us now comment on the impact of benchmark parameters NC, RF, and RS.
In Table~\ref{tab:paramG357}, we report for each value of parameter NC, RF and RS, the same entries as for Table~\ref{tab:G357} for all budgets $\Gamma =3,5,7$.\\

\begin{table}[H]
    \centering
    \footnotesize
    \begin{tabular}{r r @{\hspace{2cm}}c c c c}

&&	\#solved &	time(s) &	\#unsolved &	gap \\
\hline
NC & 1.5  & 328  & 42.83  & 152  & 18.78\%  \\
NC & 1.8  & 343  & 36.82  & 137  & 20.00\%  \\
NC & 2.1  & 355  & 54.92  & 125  & 18.65\%  \\
\hline
RF & 0.25  & 360  & 6.49  & 0  & -  \\
RF & 0.5  & 276  & 68.46  & 84  & 13.46\%  \\
RF & 0.75  & 218  & 67.38  & 142  & 20.29\%  \\
RF & 1  & 172  & 59.61  & 188  & 20.83\%  \\
\hline
RS & 0.2  & 107  & 88.05  & 253  & 27.99\%  \\
RS & 0.5  & 229  & 90.84  & 131  & 5.66\%  \\
RS & 0.7  & 330  & 42.11  & 30  & 3.48\%  \\
RS & 1  & 360  & 5.71  & 0  & -  \\

\end{tabular}
    \caption{Performance of $(\FormAdj)$ depending on PSPLib parameters for $n=30$ jobs.}
    \label{tab:paramG357}
\end{table}

\medbreak
The results show that parameter NC has a limited impact on the performance of the compact reformulation in terms of the number of instances solved, time and gap. By contrast, resource parameters RF and RS have an important impact, the hardest instances being for high RF and low RS. It corresponds to the instances where jobs use resources of different types (high RF) but instances are not highly disjunctive (low RS). Note that all instances with RF = 0.25 or RS = 1.0 are solved optimally; instances with RS $\geq$ 0.5 are solved with small final gap, 5.66\% on average.

\subsubsection{Overhead computational price of adjustable robustness}

In this section, we assess numerically the overhead computational effort that is necessary to solve the Adjustable-Robust RCPSP, in comparison with the deterministic RCPSP.
Our claim is that $(\FormAdj)$ inherits from the weakness of the flow formulation for the RCPSP.

\medbreak
Recall that the deterministic RCPSP corresponds to the case $\Gamma = 0$, and $(\FormAdj)$ coincides with the flow formulation for the RCPSP in that case. The 480 instances of the deterministic RCPSP are solved with the MIP formulation $(\FormAdj)$ and $\Gamma=0$.

In Figure~\ref{fig:solveVsTime}, the percentage of solved instances is represented depending on computation time, for the flow formulation relative to the deterministic RCPSP, and for the formulation $(\FormAdj)$ with budget $\Gamma=3,5,7$.

\begin{figure}[H]
\input{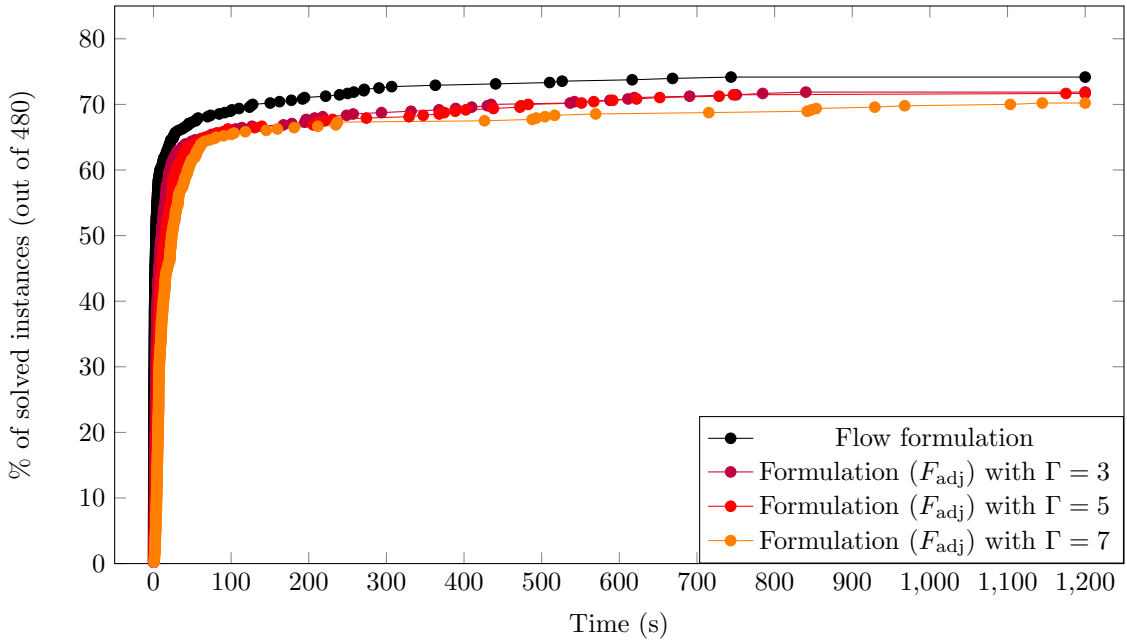}
\caption{Percentage of solved instances over time, for budgets $\Gamma=0,3,5,7$.}
\label{fig:solveVsTime}
\end{figure}

It can be observed that the proportion of solved instances is quite similar for $\Gamma=3,5,7$ or for the deterministic RCPSP. For all of them, 60\% of instances are solved within 100 seconds. Then a plateau is observed, since remaining instances are hard.
The overhead computation effort is represented by the offset of robust curves w.r.t. the deterministic one.

\subsection{Performance of $(\FormAdj)$ on larger instances}
\label{sec:res_adj_large}

In this section we provide insights on the scalability of the MIP when the number of jobs $n$ is increased.
The analysis from Section~\ref{sec:res_adj_small} led to the identification of parameters with the most influence on the performance of the MIP.
In particular instances with RF $\leq$ 0.5 and RS = 1 are the most efficiently solved for $n=30$ jobs. 
The MIP is solved for instances with such parameters and a larger number of jobs ($n = 30, 60, 90$).

In Table~\ref{tab:large} are reported the same entries as in Table~\ref{tab:G357} and a new one:\\
-- \#noSol: the number of instances where no feasible solution was found by MIP within the time limit.

\begin{table}[H]
    \footnotesize
    \centering
    \begin{tabular}{r r r r r r r}
% RS = 1 & RF <= 0.5
n & $\Gamma$ & \#solved & time(s) & \#unsolved & gap & \#noSol\\
\hline
30 & 3  & 60  & 2.04  & 0  & - & 0 \\
30 & 5  & 60  & 3.69  & 0  & - & 0 \\
30 & 7  & 60  & 5.11  & 0  & - & 0 \\
\hline
60 & 3  & 60  & 57.76  & 0  & - & 0 \\
60 & 5  & 60  & 79.08  & 0  & - & 0 \\
60 & 7  & 60  & 104.14  & 0  & - & 0 \\
\hline
90 & 3  & 59  & 272.36  & 1  & 20.26\% & 0 \\
90 & 5 & 54 & 424.80 & 6 & 25.09\% & 0 \\
90 & 7 & 42 & 471.12 & 13 & 40.46\% & 5 \\
\end{tabular}
    \caption{Performance of $(\FormAdj)$ for the 60 instances with RF $\leq$ 0.5 and RS = 1.}
    \label{tab:large}
\end{table}

While instances with $n=30$ are very efficiently solved, instances with $n=90$ with the same parameters are not all solved within the time limit of 1200 seconds. Note that for five instances with $n=90$ and $\Gamma=7$, the MIP solver does not find a feasible solution. For such instances, no optimality gap is available.
The same behavior appears on all tested instances with $n=120$ jobs, for which results are unreported.
The results show that even for instance classes that are very easy for small size ($n=30$), the MIP fails for medium-size instances.

% \newpage
\section{Computational results: Heuristics for the Adju\-stable-Robust RCPSP}
\label{sec:num:heurAdj}

In this section, numerical performance of heuristics for the Adjustable-Robust RCPSP is investigated.

\subsection{Quality evaluation of priority rules}

Let us first evaluate priority rules on instances with $n=30$ jobs. A first question is the relevance of priority rules that were designed for the deterministic RCPSP; a second one is the relevance of applying the BestRule heuristic instead of only one priority rule.
Let us denote by $\QoptGamma$ the optimal value of the Adjustable-Robust RCPSP for budget $\Gamma$. Let us denote by $\QGammaHeur$ and $\QGammaBR$ the value of heuristic $\AdjGammaHeur$ and BestRule heuristic respectively.

\subsubsection{Correlation in quality between deterministic and robust setting}

Let us address the first question: does a good heuristic $\heurAlgo$ for the RCPSP yield a good heuristic $\AdjGammaHeur$ for the Adjustable-Robust RCPSP?

Consider the instances with $n=30$ jobs where the optimal values $\QoptGamma$, $\Gamma=0,3,5,7$ have been found by solving the MIP.
The question is to assess numerically the correlation between the gaps $\GapZero(\heurAlgo) = \frac{\QZeroHeur - \QoptZero}{\QZeroHeur}$ and $\GapGamma(\heurAlgo)=\frac{\QGammaHeur - \QoptGamma}{\QGammaHeur}$, i.e., the quality of $\heurAlgo$ for the RCPSP and the quality of $\AdjGammaHeur$ for the Adjustable-Robust RCPSP.

For budgets $\Gamma=3,5,7$ we consider the data points formed with the gaps $\GapZero(\heurAlgo)$ and $\GapGamma(\heurAlgo)$ on all solved instances, for each heuristic $\heurAlgo$ based on one of the 7 priority rules presented in Section~\ref{sec:priorityRules}.
In Table~\ref{tab:pearson} the Pearson correlation coefficient between these gaps, computed with Julia function \texttt{cor()}, is reported for budgets $\Gamma=3,5,7$.
(For each value of $\Gamma$, the number of points is thus seven times the number of instances solved for both budget $\Gamma$ and budget 0.)
In Figure~\ref{fig:plotsCorrel} we represent the data points in the 2D space with first axis $\GapZero$ and second axis $\GapGamma$.

\begin{table}[H]
        \footnotesize
        \centering

        \begin{tabular}{r r @{\hspace{2cm}} c }
$\Gamma$ & \#points & Correl. Coeff. between $\GapGamma$ and $\GapZero$  \\
\hline
% cor(points[:,3], points[:,4])
3 &
2408
&
% 0.9882289578313592
0.9882
\\
5 &
2387
& 
% 0.9797079285492967 
0.9797 
\\
7 & 
2345
& 
% 0.9838327511174615
0.9838
\\
\end{tabular}
        \caption{Correlation coefficient between optimality gaps in deterministic and robust cases for instances with $n=30$ jobs.}
        \label{tab:pearson}
    \end{table}

\begin{figure}[H]
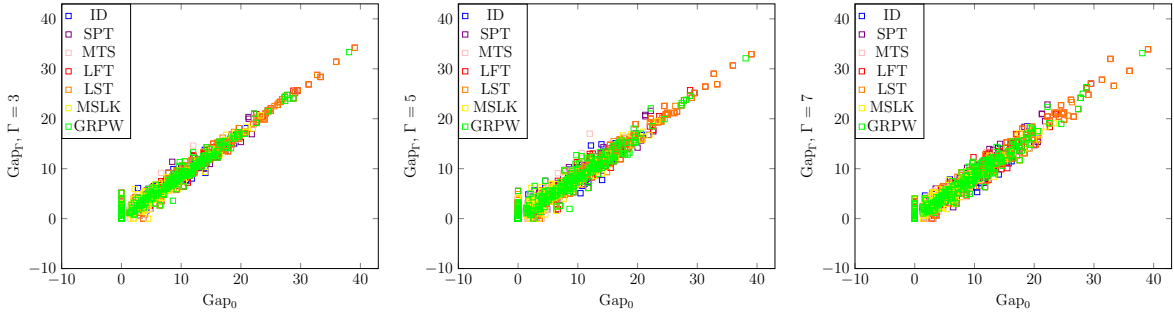

\centering
\scalebox{0.5}{
\begin{tabular}{c c c}
    \input{figures/plots/plot_correl_solved_G3}
    & \input{figures/plots/plot_correl_solved_G5}
    & \input{figures/plots/plot_correl_solved_G7}
    \\
\end{tabular}
}
\caption{ Heuristic solutions plotted in $(\GapZero, \GapGamma)$ space, for $\Gamma=3,5,7$, for instances with $n=30$ jobs.}
\label{fig:plotsCorrel}
\end{figure}

In Table~\ref{tab:pearson} correlation coefficients are around 0.98, which indicates an important correlation between $\GapZero$ and $\GapGamma$. Indeed uncorrelated data would lead to a coefficient correlation of 0; affinely related data would lead to a correlation coefficient of 1. This result justifies the use of priority rules that were studied in the literature for the minimization of the makespan in the deterministic RCPSP setting. 

\subsubsection{Relevance of BestRule}

We now investigate the second question, which is the relevance of choosing the best priority rule, rather than only one.\\
Given an instance and heuristic algorithm $\heurAlgo$, let \GapToBR denote the gap of $\heurAlgo$ to BestRule solution, defined as $\frac{\QGammaHeur - \QGammaBR}{\QGammaBR}$.
In Table~\ref{tab:heur_nbBest} are reported for each priority rule based algorithm $\heurAlgo$:\\
-- \#isBR: the number of instances (over 1440) for which the priority rule gives the best value, i.e., $\QGammaHeur = \QGammaBR$;\\
-- \GapToBR avg: the average value of \GapToBR on all instances;\\
-- \GapToBR max: the maximum value of \GapToBR on all instances.\\

  \begin{table}[H]
     \centering
        \footnotesize
        \begin{tabular}{c c c c}
Priority rule & \#isBR & \GapToBR avg & \GapToBR max \\
 \hline
 ID & 703 & +2.71\% & +37.00\% \\
 % reverse & 466 & +6.58\% & +48.27\% \\
 SPT & 540 & +4.85\% & +31.57\% \\
 MTS & 959 & +1.17\% & +13.40\% \\
 LFT & 381 & +8.38\% & +50.00\% \\
 LST & 398 & +8.32\% & +50.00\% \\
 MSLK & 866 & +2.48\% & +29.46\% \\
 GRPW & 446 & +6.27\% & +48.27\% \\
 \end{tabular}
        \caption{Quality of priority rules w.r.t. BestRule on the 1440 instances with $n=30$ jobs.}
        \label{tab:heur_nbBest}
    \end{table}

It can be observed that every priority rule is the best priority rule for some instances (at least 381 out of 1440 for LFT). Priority rules MTS and MSLK have the best results in terms of \#isBR and \GapToBR avg values. However, the makespan given by MTS is up to +13.4\% more than the BestRule makespan for some instances, as shown by \GapToBR max. The BestRule heuristic thus has a better performance in average on the benchmark.

\subsection{Performance of BestRule heuristic}

Let us now present insights on the performance of the BestRule heuristic, in terms of computation time and solution quality.\\
First the BestRule heuristic is evaluated on all instances, i.e., with $n=30,60,90,120$ jobs. 
Let BR denote the value of BestRule heuristic.
Since some of these instances are difficult to solve through the MIP formulation, the linear bound Rel (i.e., the optimal value of the continuous relaxation of the MIP) is used as a lower bound for the problem.

\medbreak
In Table~\ref{tab:heurLarge} are given, for each value of $n=30,60,90,120$ and $\Gamma=3,5,7$:\\
-- BR time: the computation time for BestRule heuristic;\\
-- Gap: the gap $\frac{BR-Rel}{BR}$ of BestRule to the linear bound Rel;\\
-- Rel time: the computation time for the linear bound Rel.\\

\begin{table}[H]
    \centering
    \footnotesize
    \begin{tabular}{r r r r r}
$n$ & $\Gamma$ & BR time (s) & Gap $\frac{BR-Rel}{BR}$ & Rel time (s) \\
\hline
30 & 3 & 0.11  & 10.33\% & 0.15  \\
30 & 5 & 0.11  & 10.05\% & 0.17  \\
30 & 7 & 0.12  & 10.37\% & 0.21  \\
\hline
60 & 3 & 0.21  & 9.54\% & 0.58  \\
60 & 5 & 0.27  & 9.24\% & 0.73  \\
60 & 7 & 0.28  & 9.15\% & 0.94  \\
\hline
90 & 3 & 0.40  & 9.18\% & 1.59  \\
90 & 5 & 0.32  & 9.03\% & 1.87  \\
90 & 7 & 0.48  & 8.97\% & 2.42  \\
\hline
120 & 3 & 0.47  & 21.62\% & 4.86  \\
120 & 5 & 0.67  & 21.00\% & 5.33  \\
120 & 7 & 0.73  & 20.66\% & 6.50  \\

\end{tabular}
    \caption{Performance of BestRule heuristic for all instances.}
    \label{tab:heurLarge}
\end{table}

The computation time of the BR heuristic is below 1 second, making it applicable to large instances with up to 120 jobs. By contrast, the time for solving the continuous relaxation of the MIP with Cplex grows faster, reaching more than 5 seconds for the instances with 120 jobs. The results provide a good signal for the scalability of the BestRule heuristic, in comparison with methods based on MIP solving.

The gap between BR and the linear bound Rel is around 10\% for instances with $n\leq 90$ jobs, but goes to more than 20\% for $n=120$ jobs. This gap is affected both by the quality of the BR heuristic and the weakness of the lower bound; in the sequel BestRule is compared to the optimal value of the problem, when known.

\vspace{1cm}

Let us now compare BestRule to solutions of $(\FormAdj)$. We distinguish between instances solved to optimality, and instances not solved to optimality where only a lower bound is known.
Let us denote by OPT the optimal value, MIPSol the value of the best solution and LB the best lower bound found by Cplex when solving $(\FormAdj)$. The gap returned by Cplex is thus $\frac{MIPSol-LB}{MIPSol}$.

\medbreak
Results are reported in Table~\ref{tab:BR_perf_30} for all $n=30$ instances, and in Table~\ref{tab:BR_perf_large_RSRF} for instances with $n=30,60,90$ and RF $\leq$ 0.5, RS = 1 for which MIP results were previously analyzed.
In both tables, we report for budget $\Gamma=3,5,7$:\\
-- for all instances:\\
\hspace*{1cm} -- BR gap: the gap $\frac{BR - LB}{BR}$ of BestRule to lower bound LB;\\
\hspace*{1cm} -- MIP gap: the gap $\frac{MIPSol - LB}{MIPSol}$ obtained by MIP solving;\\
-- for instances solved optimally by MIP:\\
\hspace*{1cm} -- \#instances: the number of such instances;\\
\hspace*{1cm} -- BR gap: the gap $\frac{BR - OPT}{BR}$ of BestRule to optimal value OPT;\\
\hspace*{1cm} -- \#(BR=OPT): the number of instances where BestRule heuristic returns an optimal solution.\\
-- for instances not solved optimally by MIP:\\
\hspace*{1cm} -- BR gap: the gap $\frac{BR - LB}{BR}$ of BestRule to lower bound LB;\\
\hspace*{1cm} -- MIP gap: the gap $\frac{MIPSol - LB}{MIPSol}$ obtained by MIP solving;\\
\hspace*{1cm} -- \#(BR$<$MIPSol) : the number of instances where BestRule heuristic returns a solution that is better than the best incumbent found by MIP;\\
\hspace*{1cm} -- \#noMIPSol : the number of instances where no incumbent was found by MIP.\\

\begin{table}[H]
	\begin{adjustwidth}{-2cm}{-2cm}
    \centering
    \footnotesize
    
\begin{tabular}{r r r r r r c r r r c c }
 && \multicolumn{2}{c}{all instances} &
 \multicolumn{3}{c}{instances solved by MIP} & 
 \multicolumn{4}{c}{instances not solved by MIP} \\
$n$ & $\Gamma$ & BR gap & MIP gap & \#inst & BR gap & \#(BR=OPT) & \#inst & BR gap & MIP gap & \#(BR$<$MIPSol) & \#noMIPSol\\ 
 \hline
30 & 3  & 7.42\% &  5.41\%  & 341 & 1.49 \% & 201  & 139 & 21.97 \% & 18.71 \% & 17 & 0\\
30 & 5  & 7.36\% &  5.38\%  & 342 & 1.49 \% & 196  & 138 & 21.89 \% & 18.74 \% & 23 & 0\\
30 & 7  & 7.54\% &  5.71\%  & 335 & 1.42 \% & 199  & 145 & 21.69 \% & 18.90 \% & 27  & 0\\
\end{tabular}
    \caption{Performance of BestRule heuristic for instances with $n=30$ jobs.}
    \label{tab:BR_perf_30}
	\end{adjustwidth}
\end{table}

\begin{table}[H]
	\begin{adjustwidth}{-2cm}{-2cm}
    \centering
    \footnotesize
    
\begin{tabular}{r r r r r r c r r r c c}
 && \multicolumn{2}{c}{all instances} &
 \multicolumn{3}{c}{instances solved by MIP} & 
 \multicolumn{4}{c}{instances not solved by MIP} \\
$n$ & $\Gamma$ & BR gap & MIP gap & \#inst & BR gap & \#(BR=OPT) & \#inst & BR gap & MIP gap & \#(BR$<$MIPSol) & \#noMIPSol\\ 
 \hline
30 & 3  & 0.14\% &  0.0\%  & 60 & 0.14 \% & 55  & 0 & - & - & - & 0 \\
30 & 5  & 0.16\% &  0.0\%  & 60 & 0.16 \% & 55  & 0 & - & - & - & 0 \\
30 & 7  & 0.08\% &  0.0\%  & 60 & 0.08 \% & 56  & 0 & - & - & - & 0 \\
\hline
60 & 3  & 0.06\% &  0.0\%  & 60 & 0.06 \% & 57  & 0 & - & - & - & 0 \\
60 & 5  & 0.08\% &  0.0\%  & 60 & 0.08 \% & 55  & 0 & - & - & - & 0 \\
60 & 7  & 0.12\% &  0.0\%  & 60 & 0.12 \% & 50  & 0 & - & - & - & 0 \\
\hline
90 & 3  & 0.08\% &  0.33\%  & 59 & 0.07 \% & 51  & 1 & 0.54 \% & 20.26 \% & 1 & 0 \\
90 & 5  & 0.19\% &  2.50\%  & 54 & 0.17 \% & 45  & 6 & 0.39 \% & 25.09 \% & 6 & 0 \\
90 & 7  & 0.27\% &  9.56\%  & 42 & 0.23 \% & 33  & 13 & 0.39 \% & 40.46 \% & 13 & 5 \\

\end{tabular}
    \caption{Performance of BestRule heuristic for instances with RF $\leq$ 0.5, RS = 1 for $n=30, 60, 90$ jobs.}
    \label{tab:BR_perf_large_RSRF}
    \end{adjustwidth}
\end{table}

In Table~\ref{tab:BR_perf_30} it can be observed that the BR gap is at most 7.54\% on average; BR gap is around 2\% larger than the MIP gap.
For instances where OPT is known, BR gap is even smaller (at most 1.49\%). Notably, BestRule outputs an optimal solution on a large subset of instances: 596 out of 1440.
For instances where OPT is not known, the BR gap is also close to MIP gap. On 67 instances, BestRule outputs a solution which is better than the best incumbent found by MIP solving after 1200 seconds computation.

In Table~\ref{tab:BR_perf_large_RSRF}, similar results are obtained
for instances with $n=30, 60$ jobs, where OPT is known for all instances. The BR gap is very small (at most 0.14\% on average), and BestRule gives an optimal solution for 328 instances out of 360.
For $n=90$ jobs, not all instances are solved by MIP. 
The BR gap is still very small for instances solved (at most 0.23\% on average), and BestRule gives an optimal solution for 130 instances out of 180.
For instances where OPT is not known, the BestRule heuristic is very efficient (BR gap $\leq$ 1\%), while MIP gap is large (more than 20\%). The BestRule is better than the best MIP solution on all these instances. Note also that for $\Gamma=7$, there are 5 instances where no MIP solution was known, while BestRule does provide a solution. 
The good performance of BestRule on unsolved instances results into a very good performance on average on all instances: e.g., for $\Gamma=7$, BR gap is only 0.27\% versus a MIP gap of 9.56\%.

\section{Computational results: MIP for Anchor-Robust RCPSP}
\label{sec:num:exactAnch}

In this section, the compact MIP reformulation $(\FormAnch)$ for Anchor-Robust RCPSP is evaluated. The impact of deadline $M$ on the number of anchored jobs is also discussed.

\subsection{Instances and settings}

Let us describe the settings of numerical experiments.\\
The Anchor-Robust RCPSP instances are defined as follows.
The instances of the RCPSP and uncertainty sets are the same as in Section~\ref{sec:num:exactAdj}. The only additional input of the Anchor-Robust RCPSP is the deadline $M$. As proposed in Section~\ref{sec:models}, the deadline $M$ can be found by solving the Adjustable-Robust RCPSP. Since the Adjustable-Robust RCPSP is difficult to solve even for some instances with $n=30$ jobs, we consider in the sequel that deadline $M$ is set to the value $\deadlineBR = \QGamma(BR)$ of BestRule heuristic.
The BestRule heuristic is thus applied on each instance to determine $\deadlineBR$. The solution found by BestRule heuristic is a sequencing decision $\sigmaHeur$ with worst-case makespan $\QGamma(\sigmaHeur) = \deadlineBR$. The sequencing decision $\sigmaHeur \in \Sigmas$ is fed to $(\FormAnch)$ as a warm-start. Note that this gives a feasible solution to the MIP solver, with $h_i=0$ for every $i \in J$.

Computational settings are the same as in Section~\ref{sec:num:exactAdj}. The time limit for MIP solving is set to 1200 seconds. The computation times reported in this section do not include the computation time of BestRule heuristic, which is negligible w.r.t. MIP solving.

\subsection{Performance of $(\FormAnch)$ on small instances}

Let us first assess the performance of the compact reformulation $(\FormAnch)$ for instances with $n=30$ jobs, and analyze the influence of parameters.

In Table~\ref{tab:anch_budget} and Table~\ref{tab:anch_param} are reported, for each value of $\Gamma=3,5,7$ and for each value of parameter NC, RF and RS respectively:\\
-- Anch: the average number of anchored jobs in the best solution found by $(\FormAnch)$;\\
% -- UB: the average best upper bound;\\
-- \#solved: the number of instances solved to optimality within time limit;\\
-- time: the computation time averaged on instances solved to optimality, in seconds;\\
-- \#unsolved: the number of instances not solved to optimality within time limit;\\
-- gap: the final gap averaged on instances not solved to optimality. It equals $\frac{UB - Anch}{Anch}$ where $UB$ is the best upper bound found by MIP;\\
-- \gapn: the final gap $\frac{UB - Anch}{n}$ expressed as a percentage of the number of jobs.\\

\begin{table}[H]
    \centering
    \footnotesize
    \begin{tabular}{r r@{\hspace{1cm}} r r r r r r r}

$\Gamma$ & Anch &  \#solved & time (s) & \#unsolved & gap & \gapn\\
\hline
 3  & 16.09  & 307  & 32.76  & 173 & 244.86\%  & 53.78\%  \\
 5  & 17.92  & 321  & 25.42  & 159 & 261.78\%  & 55.10\%  \\
 7  & 23.01  & 372  & 34.69  & 108 & 237.08\%  & 56.08\%  \\
\hline
 all  & 19.01  & 1000  & 31.12  & 440 & 249.06\%  & 54.82\%  \\

\end{tabular}
    \caption{Performance of $(\FormAnch)$ for $n=30$ jobs, w.r.t. budget $\Gamma$.}
    \label{tab:anch_budget}
\end{table}

\begin{table}[H]
    \centering
    \footnotesize
    \begin{tabular}{r r r @{\hspace{1cm}} r r r r r r}

& & Anch  & \#solved & time (s) & \#unsolved & gap & \gapn\\
\hline
NC & 1.5  & 20.63  & 340  & 28.97  & 140 & 223.59\%  & 53.53\%  \\
NC & 1.8  & 18.76  & 332  & 27.09  & 148 & 272.88\%  & 56.24\%  \\
NC & 2.1  & 17.62  & 328  & 37.44  & 152 & 249.34\%  & 54.63\%  \\
\hline
RF & 0.25  & 21.58  & 359  & 1.83  & 1 & 25.82\%  & 16.35\%  \\
RF & 0.5  & 19.88  & 261  & 32.11  & 99 & 126.86\%  & 48.33\%  \\
RF & 0.75  & 17.82  & 199  & 55.69  & 161 & 243.28\%  & 51.99\%  \\
RF & 1  & 16.74  & 181  & 60.8  & 179 & 323.1\%  & 61.17\%  \\
\hline
RS & 0.2  & 12.67  & 98  & 36.53  & 262 & 356.03\%  & 68.84\%  \\
RS & 0.5  & 20.12  & 232  & 74.92  & 128 & 106.01\%  & 38.43\%  \\
RS & 0.7  & 21.25  & 310  & 24.66  & 50 & 54.78\%  & 23.31\%  \\
RS & 1  & 21.98  & 360  & 6.99  & 0 & - & -  \\
\end{tabular}
    \caption{Performance of $(\FormAnch)$ for $n=30$, w.r.t. PSPLib parameters.}
    \label{tab:anch_param}
\end{table}

Let us first comment on the optimal value Anch of the problem. As reported in Table~\ref{tab:anch_budget}, on average around 19 jobs out of 30 can be anchored in a baseline schedule within the makespan $\deadlineBR$.
This highlights that a large proportion of jobs can be anchored, while respecting the worst-case makespan $\deadlineBR$ obtained in first place. The impact of the choice for deadline $\deadlineBR$ is discussed in Section~\ref{sec:deadline}.

\medbreak
Let us now comment on performance of the compact reformulation.
First $(\FormAnch)$ allows for solving 1000 instances out of 1440 to optimality. For instances that are not solved to optimality, the value Anch can be very small, which results into a very large gap (264\% on average). Recall that Anch appears as denominator in the gap formula. The \gapn\ value is more suitable for interpretation. The average \gapn\ value is 54.82\%, thus showing the poor performance of $(\FormAnch)$ on unsolved instances.

Table~\ref{tab:anch_budget} shows that the uncertainty budget has an impact on the number of instances solved: by contrast with the Adjustable-Robust RCPSP, for high budget $\Gamma=7$, more instances are solved to optimality. This is interpreted in Section~\ref{sec:anch_large}.

Regarding the impact of PSPLib parameters reported in Table~\ref{tab:anch_param}, it appears that the same instance subsets are hard for both the Anchor-Robust RCPSP and the Adjustable-Robust RCPSP. Namely, all instances with RF = 0.25 except one are solved to optimality, and all instances with RS = 1 are solved to optimality. By contrast parameter NC has no significant impact on performance.

\subsection{Performance of $(\FormAnch)$ on larger instances}
\label{sec:anch_large}

Let us now examine the scalability of the MIP formulation $(\FormAnch)$ when the number of jobs $n$ is increased.
Similarly to Section~\ref{sec:res_adj_large}, we restrict our attention to instances with RF $\leq$ 0.5 and RS = 1 that are efficiently solved for $n=30$ jobs. The MIP formulation $(\FormAnch)$ is solved for instances with RF $\leq$ 0.5, RS = 1 and $n = 30, 60, 90$ jobs. In Table~\ref{tab:anch_large} are reported the same entries as in Table~\ref{tab:anch_budget}.

\begin{table}[H]
    \centering
    \footnotesize
    \begin{tabular}{r r r r r r r r}

& & Anch & \#solved & time (s) & \#unsolved & gap & \gapn\\
\hline
30 & 3  & 19.78  & 60  & 0.25  & 0  & -  \\
30 & 5  & 21.28  & 60  & 0.51  & 0  & -  \\
30 & 7  & 24.48  & 60  & 0.47  & 0  & -  \\
\hline
60 & 3  & 38.11  & 57  & 16.79  & 3  & 6.13\%  \\
60 & 5  & 40.38  & 59  & 25.17  & 1  & 2.77\%  \\
60 & 7  & 44.25  & 59  & 6.00  & 1  & 4.76\%  \\
\hline
90 & 3  & 61.68  & 54  & 54.33  & 6  & 7.46\%  \\
90 & 5  & 65.41  & 59  & 76.28  & 1  & 4.91\%  \\
90 & 7  & 70.61  & 60  & 76.30  & 0  & -  \\

\end{tabular}
    \caption{Performance of $(\FormAnch)$ for instances with RF $\leq$ 0.5 and RS = 1.}
    \label{tab:anch_large}
\end{table}

For $n=60$ and $n=90$ jobs, some instances are not solved to optimality within 1200 seconds. Note that for the Adjustable-Robust RCPSP, all instances with $n=60$ jobs were solved, as reported previously in Table~\ref{tab:large}.
Computation time for $n=90$ instances is lower than that for the Adjustable-Robust RCPSP on the same instances. This is due to the warm-start, which gives a feasible sequencing decision matching the deadline to $(\FormAnch)$. The warm-start prevents $(\FormAnch)$ from finding no feasible solution.

Contrary to the Adjustable-Robust RCPSP, instances for larger $\Gamma$ are solved more easily. In particular, the number of unsolved instances decreases with $\Gamma$.
This can be surprising since the size of the MIP formulation increases with $\Gamma$. Note that the average number of anchored jobs Anch is higher for $\Gamma=7$. Hence a large part of anchoring variables can be quickly set to one by the solver, and a small subset of anchoring variables is left for optimization. This yields simpler instances.

\subsection{Impact of the deadline}
\label{sec:deadline}

Let us finally discuss the impact of the chosen deadline $M$. We proposed to choose as deadline the output $\deadlineBR$ of the BestRule heuristic, and we observed the number of jobs that can be anchored. The deadline may also be increased, so that more jobs can be anchored. A question is to evaluate the impact of an increase of the deadline w.r.t. the number of jobs that can be anchored.

To that end, the Anchor-Robust RCPSP is solved for deadline $\beta \deadlineBR$, with $\beta$ a scaling factor ranging in $\{1.0, 1.05, 1.1,$ $1.2, 1.3\}$. Instances with RF = 0.25, RS = 1 are considered, for which the optimal value of the Anchor-Robust RCPSP is efficiently computed. In Table~\ref{tab:anch_deadline} are reported, for each value of the deadline $\beta \deadlineBR$, the optimal number of anchored jobs. Each line corresponds to an average over the 30 instances, for budget $\Gamma=3$. The last line corresponds to the average for $n=30,60,90$, the value being given as a percentage of anchored jobs.

\begin{table}[H]
    \centering
    \footnotesize
    \begin{tabular}{c @{\hspace{1cm}}c c c c c}
 & \multicolumn{5}{c}{Scaling factor $\beta$ applied to deadline} \\
$n$ & 1.0 & 1.05 & 1.1 & 1.2 & 1.3 \\ 
\hline
 30  & 18.39  & 21.40  & 24.26  & 28.73  & 30.0  \\
 60  & 39.36  & 44.60  & 49.10  & 55.56  & 59.9  \\
 90  & 59.83  & 67.63  & 74.76  & 83.73  & 89.2  \\
 \hline
all  & 64.47\%  & 73.6\%  & 81.93\%  & 93.8\%  & 99.64\%  \\

\end{tabular}
    \caption{Optimal number of anchored jobs depending on the deadline $\beta \deadlineBR$, for $\Gamma=3$ and RF = 0.25, RS = 1.}
    \label{tab:anch_deadline}
\end{table}

As noted previously, a large proportion of jobs can be anchored even for deadline $\deadlineBR$: around 65\% of jobs. A slight increase of the deadline, e.g., an increase of +5\%, allows for an increase of +10\% of the number of anchored jobs. Finally, for a large increase of the makespan (+30\%) on almost all instances all jobs are anchored, meaning that deadline 1.3$\deadlineBR$ is more than the makespan of a static-robust solution.

This highlights the interest of using Anchor-Robust RCPSP to find a baseline schedule, with the opportunity of tuning the deadline value. A trade-off can be found between the deadline and the number of anchored jobs.

\section{Computational results: Heuristic for Anchor-Robust RCPSP}
\label{sec:num:heurAnch}

In this section, we evaluate the BestRuleSequence (BRS) heuristic to solve hard instances of the Anchor-Robust RCPSP.

\subsection{Scalability of BRS heuristic}

We first evaluate the performance of BRS heuristic on larger instances. The MIP for BRS is solved under the same computational settings as $(\FormAnch)$ studied in Section~\ref{sec:num:exactAnch}. The fixed sequencing decision is $\sigmaHeur$ given by BestRule heuristic, and the deadline is $\deadlineBR = \QGamma(\sigmaHeur)$. Recall that the MIP formulation for BRS heuristic does not have any binary variable for the sequencing decision, but only anchoring binary variables and continuous schedule variables.

In Table~\ref{tab:anch_BRS_large} are reported, for each value of $n = 30, 60, 90, 120$ and $\Gamma = 3,5,7$:\\
-- AnchBRS: the average number of anchored jobs obtained by BRS;\\
-- AnchBRS(\%): AnchBRS expressed as a percentage of $n$;\\
-- BR time: the average time for BestRule heuristic, in seconds;\\
-- BRS time: the average time for BestRuleSequence heuristic, in seconds;\\
-- Total time: the sum of BR time and BRS time, in seconds.

\begin{table}[H]
    \centering
    \footnotesize
    \begin{tabular}{r r @{\hspace{1cm}} r r r r r}
$n$ & $\Gamma$ & AnchBRS & AnchBRS(\%) & BR time & BRS time & Total time \\
\hline
30 & 3  & 9.19 & 30.63\%  & 0.11  & 0.01  & 0.13  \\
30 & 5  & 11.47 & 38.26\%  & 0.11  & 0.01  & 0.12  \\
30 & 7  & 16.65 & 55.51\%  & 0.12  & 0.01  & 0.13  \\
\hline
60 & 3  & 13.01 & 21.69\%  & 0.21  & 0.01  & 0.22  \\
60 & 5  & 14.75 & 24.58\%  & 0.27  & 0.01  & 0.28  \\
60 & 7  & 19.26 & 32.10\%  & 0.28  & 0.01  & 0.30  \\
\hline
90 & 3  & 16.39 & 18.22\%  & 0.40  & 0.01  & 0.42  \\
90 & 5  & 17.72 & 19.69\%  & 0.32  & 0.02  & 0.35  \\
90 & 7  & 21.49 & 23.88\%  & 0.48  & 0.02  & 0.51  \\
\hline
120 & 3  & 12.30 & 10.25\%  & 0.47  & 0.01  & 0.49  \\
120 & 5  & 12.74 & 10.61\%  & 0.67  & 0.02  & 0.69  \\
120 & 7  & 14.75 & 12.29\%  & 0.73  & 0.03  & 0.76  \\
\end{tabular}
    \caption{Performance of BRS heuristic for all instances.}
    \label{tab:anch_BRS_large}
\end{table}

It can be observed first, that even if BRS is based on solving an MIP, it is solved extremely efficiently with Cplex: the order of magnitude is 10 milliseconds. The program is often solved at presolve or root node.
A connection can be made with results for exact solving of the Anchor-Robust RCPSP presented in Section~\ref{sec:num:exactAnch}.
This shows that the computational effort necessary to solve Anchor-Robust RCPSP is mainly due to sequencing decision variables: for fixed sequencing decision the problem is very easily handled through MIP.

Total time (to run BestRule heuristic and find $\deadlineBR$, then run BRS heuristic to find an associated baseline schedule and anchored jobs) is less than a second on all considered instances, proving the scalability of BRS heuristic.

Finally the number of anchored jobs found with BRS is high for small instances: up to 55\% for $n=3$ and $\Gamma=7$. It drops to 10\% for instances with $n=120$ jobs. This raises the question of evaluating BRS heuristic w.r.t. the optimal value of the Anchor-Robust RCPSP.

\subsection{Comparison to exact solution of the Anchor-Robust RCPSP}

Let us compare the value AnchBRS obtained by BRS heuristic to the optimal value of the Anchor-Robust RCPSP. Let Anch denote the best value found by MIP, and UB the best upper bound found by MIP.

Results are reported in Table~\ref{tab:BR_perf_30} for all $n=30$ instances, and in Table~\ref{tab:BR_perf_large_RSRF} for instances with $n=30,60,90$ and RF $\leq$ 0.5, RS = 1 for which MIP results were previously analyzed.
In both tables, we report for budget $\Gamma=3,5,7$:\\
-- for all instances:\\
\hspace*{1cm} -- BRS \gapn: the gap $\frac{UB - AnchBRS}{n}$ of BestRuleSequence to upper bound UB;\\
\hspace*{1cm} -- MIP \gapn: the gap $\frac{UB - MIPSol}{n}$ obtained by MIP solving;\\
-- for instances solved optimally by MIP:\\
\hspace*{1cm} -- \#instances: the number of such instances;\\
\hspace*{1cm} -- BRS \gapn: the gap $\frac{OPT - AnchBRS}{n}$ of BestRuleSequence to optimal value OPT;\\
\hspace*{1cm} -- \#(BRS=OPT): the number of instances where BestRuleSequence heuristic returns an optimal solution.\\
-- for instances not solved optimally by MIP:\\
\hspace*{1cm} -- BRS \gapn: the gap $\frac{UB - AnchBRS}{n}$ of BestRuleSequence to upper bound UB;\\
\hspace*{1cm} -- MIP \gapn: the gap $\frac{UB - MIPSol}{n}$ obtained by MIP solving.\\

\begin{table}[H]
	\begin{adjustwidth}{-1.5cm}{-1.5cm}    
    \centering
    \footnotesize
    %\scriptsize
    \begin{tabular}{r r r r r r c r r r}
 && \multicolumn{2}{c}{all instances} &
 \multicolumn{3}{c}{instances solved by MIP} & 
 \multicolumn{3}{c}{instances not solved by MIP} \\
$n$ & $\Gamma$ & BRS \gapn & MIP \gapn & \#inst & BRS \gapn & \#(BRS=OPT) & \#inst & BRS \gapn & MIP \gapn \\ 
 \hline
30 & 3  & 42.39\% &  19.38\%  & 307 & 25.38 \% & 0  & 173 & 72.57 \% & 53.78 \%  \\
30 & 5  & 39.73\% &  18.25\%  & 321 & 21.48 \% & 48  & 159 & 76.57 \% & 55.10 \%  \\
30 & 7  & 33.81\% &  12.61\%  & 372 & 20.51 \% & 136  & 108 & 79.62 \% & 56.08 \%  \\ 
 
 \end{tabular}
    \caption{Performance of BRS heuristic for instances with $n=30$ jobs.}
    \label{tab:anch_BRS_perf_30}
    \end{adjustwidth}
\end{table}

\begin{table}[H]
	\begin{adjustwidth}{-1.5cm}{-1.5cm}    
    \centering
    \footnotesize
    %\scriptsize
    \begin{tabular}{r r r r r r c r r r}
 && \multicolumn{2}{c}{all instances} &
 \multicolumn{3}{c}{instances solved by MIP} & 
 \multicolumn{3}{c}{instances not solved by MIP} \\
$n$ & $\Gamma$ & BRS \gapn & MIP \gapn & \#inst & BRS \gapn & \#(BRS=OPT) & \#inst & BRS \gapn & MIP \gapn \\ 
\hline
30 & 3  & 17.05\% &  0\%  & 60 & 17.05 \% & 0  & 0 & - & -  \\
30 & 5  & 7.44\% &  0\%  & 60 & 7.44 \% & 23  & 0 & - & -  \\
30 & 7  & 3.05\% &  0\%  & 60 & 3.05 \% & 40  & 0 & - & -  \\
\hline
60 & 3  & 24.33\% &  0.16\%  & 57 & 23.53 \% & 1  & 3 & 39.44 \% & 3.33 \%  \\
60 & 5  & 23.19\% &  0.02\%  & 59 & 22.99 \% & 2  & 1 & 35.00 \% & 1.66 \% \\
60 & 7  & 19.88\% &  0.05\%  & 59 & 19.80 \% & 8  & 1 & 25.00 \% & 3.33 \%  \\
\hline
90 & 3  & 31.72\% &  0.42\%  & 54 & 30.32 \% & 0  & 6 & 44.25 \% & 4.25 \%  \\
90 & 5  & 32.98\% &  0.05\%  & 59 & 32.54 \% & 0  & 1 & 58.88 \% & 3.33 \%  \\
90 & 7  & 29.77\% &  0\%  & 60 & 29.77 \% & 0  & 0 & - & -  \\
 
 \end{tabular}
    \caption{Performance of BRS heuristic for instances with RF $\leq$ 0.5, RS = 1.}
    \label{tab:anch_BRS_perf_RSRF}
    \end{adjustwidth}
\end{table}

Table~\ref{tab:anch_BRS_perf_30} shows that BRS \gapn\ is very large, e.g., 42.39\% for instances with $n=30$ jobs and $\Gamma=3$. The gap is also large for instances solved to optimality, although the BRS heuristic provides an optimal solution to the Anchor-Robust RCPSP for 184 instances out of 1440.

Table~\ref{tab:anch_BRS_perf_RSRF} shows that the BRS heuristic has good performance for small instances with easy RF and RS: namely BRS \gapn\ is 3.05\% for instances with $n=30$ and $\Gamma=7$. For larger instances, BRS \gapn\ is large although the MIP \gapn\ is very small.

Hence the performance in terms of solution quality of BRS heuristic is limited, except for small, easy instances, and high budget.
Note that it implies the following on the Anchor-Robust RCPSP. While the worst-case makespan $\deadlineBR$ may be associated to a sequencing decision $\sigmaHeur$, there exists other sequence decisions respecting the deadline $\deadlineBR$, and such that many more jobs can be anchored. Hence, restricting to the sequencing decision $\sigmaHeur$ leads to BRS heuristic, which has poor quality but is very efficient in terms of computation time.

\newpage
\section{Conclusion}

In the present work, the new concept of anchor-robustness for the RCPSP was introduced. The definition of anchored solutions encapsulates the definition of solutions of the adjustable-robust approach. It offers a guarantee of starting times, in addition to a guarantee of the worst-case makespan. Anchor-robustness bridges the gap between static robustness and adjustable robustness from the literature.

We extended a graph model designed for PERT scheduling and obtained results for the Anchor-Robust and Adjustable-Robust RCPSP, leading to both exact and heuristic approaches.
Regarding exact solution approaches, we obtained the existence of compact reformulations for both problems.
For heuristics, we showed how to benefit from efficient heuristics for the RCPSP to design efficient heuristics for Anchor-Robust and Adjustable-Robust RCPSP. Altogether this is a complete toolbox for solving the Anchor-Robust and the Adjustable-Robust RCPSP.

A numerical evaluation of the proposed tools was performed. The compact reformulation for the Adjustable-Robust RCPSP is competitive with decomposition algorithms from the literature. We studied the scalability of MIPs, and identified hard instances. It turns out that these are the same as for the RCPSP, where optimizing the resource sequencing decision is computationally challenging.
To address such hard instances, the proposed heuristics are very efficient. 
The combination of BestRule and BestRuleSequence heuristics provides a robust makespan and baseline schedule with anchored jobs.

\medbreak
An important perspective is to improve MIP formulations. As illustrated in numerical experiments, the computational difficulty mainly arises from the binary variables for the sequencing decision.
A question is to improve the flow formulation for the RCPSP, e.g., with valid inequalities. This would surely improve the compact reformulations we proposed for the Adjustable-Robust and Anchor-Robust RCPSP.
The efficiency of heuristics also raises the question of  combining heuristics and MIP solving, beyond the mere warm-start we proposed.
The anchor-robust approach can also be extended to allow for sequencing decision revision in second stage. The sequencing decision may be partially adapted after uncertainty is known. The question of finding exact and heuristic approaches for this new problem would be an interesting research perspective.

\appendix
\section{Extended computational results}

Let us present complete computational results for solving the Adjustable-Robust RCPSP with MIP formulation $(\FormAdj)$.
In Tables~\ref{tab:G3},\ref{tab:G5},\ref{tab:G7} are reported the results for $\Gamma = 3,5,7$ respectively. For each instance class j30x, x $\in \{1, \dots, 48\}$, we indicate for $(\FormAdj)$:\\
-- the corresponding parameters NC, RF, RS;\\
-- solved: the number of instances solved to optimality (out of 10);\\
-- time: the computation time averaged on instances solved to optimality in seconds;\\
-- unsolved = 10 - solved\\
-- gap: the final gap, averaged on instances not solved to optimality.

\begin{table}[H]
    \centering
    \footnotesize
    \begin{tabular}{r r r r@{\hspace{1cm}}c r c r}
% res_adjustable_G3_1a48_1a10_ALL
class & NC & RF & RS & solved & time(s) & unsolved & gap \\
\hline
1 & 1.5 & 0.25 & 0.2 & 10  & 7.21  & 0  & -  \\
2 & 1.5 & 0.25 & 0.5 & 10  & 4.34  & 0  & -  \\
3 & 1.5 & 0.25 & 0.7 & 10  & 2.36  & 0  & -  \\
4 & 1.5 & 0.25 & 1 & 10  & 1.43  & 0  & -  \\
5 & 1.5 & 0.5 & 0.2 & 0  & -  & 10  & 16.43\%  \\
6 & 1.5 & 0.5 & 0.5 & 9  & 151.6  & 1  & 1.52\%  \\
7 & 1.5 & 0.5 & 0.7 & 10  & 9.44  & 0  & -  \\
8 & 1.5 & 0.5 & 1 & 10  & 2.56  & 0  & -  \\
9 & 1.5 & 0.75 & 0.2 & 0  & -  & 10  & 31.38\%  \\
10 & 1.5 & 0.75 & 0.5 & 3  & 198.88  & 7  & 4.94\%  \\
11 & 1.5 & 0.75 & 0.7 & 8  & 146.67  & 2  & 1.76\%  \\
12 & 1.5 & 0.75 & 1 & 10  & 5.31  & 0  & -  \\
13 & 1.5 & 1 & 0.2 & 0  & -  & 10  & 36.13\%  \\
14 & 1.5 & 1 & 0.5 & 1  & 23.14  & 9  & 6.69\%  \\
15 & 1.5 & 1 & 0.7 & 9  & 39.23  & 1  & 6.0\%  \\
16 & 1.5 & 1 & 1 & 10  & 4.76  & 0  & -  \\
17 & 1.8 & 0.25 & 0.2 & 10  & 4.15  & 0  & -  \\
18 & 1.8 & 0.25 & 0.5 & 10  & 2.03  & 0  & -  \\
19 & 1.8 & 0.25 & 0.7 & 10  & 1.5  & 0  & -  \\
20 & 1.8 & 0.25 & 1 & 10  & 0.94  & 0  & -  \\
21 & 1.8 & 0.5 & 0.2 & 1  & 587.97  & 9  & 10.76\%  \\
22 & 1.8 & 0.5 & 0.5 & 9  & 61.41  & 1  & 1.47\%  \\
23 & 1.8 & 0.5 & 0.7 & 10  & 14.33  & 0  & -  \\
24 & 1.8 & 0.5 & 1 & 10  & 3.69  & 0  & -  \\
25 & 1.8 & 0.75 & 0.2 & 0  & -  & 10  & 31.53\%  \\
26 & 1.8 & 0.75 & 0.5 & 7  & 193.15  & 3  & 2.36\%  \\
27 & 1.8 & 0.75 & 0.7 & 10  & 11.97  & 0  & -  \\
28 & 1.8 & 0.75 & 1 & 10  & 4.71  & 0  & -  \\
29 & 1.8 & 1 & 0.2 & 0  & -  & 10  & 39.4\%  \\
30 & 1.8 & 1 & 0.5 & 0  & -  & 10  & 5.74\%  \\
31 & 1.8 & 1 & 0.7 & 8  & 42.99  & 2  & 4.14\%  \\
32 & 1.8 & 1 & 1 & 10  & 10.92  & 0  & -  \\
33 & 2.1 & 0.25 & 0.2 & 10  & 3.2  & 0  & -  \\
34 & 2.1 & 0.25 & 0.5 & 10  & 1.71  & 0  & -  \\
35 & 2.1 & 0.25 & 0.7 & 10  & 1.36  & 0  & -  \\
36 & 2.1 & 0.25 & 1 & 10  & 0.78  & 0  & -  \\
37 & 2.1 & 0.5 & 0.2 & 6  & 361.35  & 4  & 17.97\%  \\
38 & 2.1 & 0.5 & 0.5 & 10  & 55.07  & 0  & -  \\
39 & 2.1 & 0.5 & 0.7 & 10  & 5.82  & 0  & -  \\
40 & 2.1 & 0.5 & 1 & 10  & 2.83  & 0  & -  \\
41 & 2.1 & 0.75 & 0.2 & 0  & -  & 10  & 26.36\%  \\
42 & 2.1 & 0.75 & 0.5 & 6  & 66.56  & 4  & 4.55\%  \\
43 & 2.1 & 0.75 & 0.7 & 10  & 155.76  & 0  & -  \\
44 & 2.1 & 0.75 & 1 & 10  & 4.2  & 0  & -  \\
45 & 2.1 & 1 & 0.2 & 0  & -  & 10  & 35.24\%  \\
46 & 2.1 & 1 & 0.5 & 2  & 483.59  & 8  & 6.51\%  \\
47 & 2.1 & 1 & 0.7 & 6  & 81.55  & 4  & 4.19\%  \\
48 & 2.1 & 1 & 1 & 10  & 4.33  & 0  & -  \\
% ALL: &&& 345 & 39.53 & 135 & 19.26 % \\
\end{tabular}
    \caption{Performance of the compact reformulation $(\FormAdj)$ for $\Gamma=3$.}
    \label{tab:G3}
\end{table}

\begin{table}[H]
    \centering
    \footnotesize
    \begin{tabular}{r r r r@{\hspace{1cm}}c r c r}
% res_adjustable_G3_1a48_1a10_ALL
class & NC & RF & RS & solved & time(s) & unsolved & gap \\
\hline

1 & 1.5 & 0.25 & 0.2 & 10  & 17.98  & 0  & -  \\
2 & 1.5 & 0.25 & 0.5 & 10  & 9.45  & 0  & -  \\
3 & 1.5 & 0.25 & 0.7 & 10  & 4.16  & 0  & -  \\
4 & 1.5 & 0.25 & 1 & 10  & 2.23  & 0  & -  \\
5 & 1.5 & 0.5 & 0.2 & 0  & -  & 10  & 15.83\%  \\
6 & 1.5 & 0.5 & 0.5 & 8  & 162.29  & 2  & 1.62\%  \\
7 & 1.5 & 0.5 & 0.7 & 10  & 45.57  & 0  & -  \\
8 & 1.5 & 0.5 & 1 & 10  & 6.0  & 0  & -  \\
9 & 1.5 & 0.75 & 0.2 & 0  & -  & 10  & 29.86\%  \\
10 & 1.5 & 0.75 & 0.5 & 4  & 184.07  & 6  & 5.72\%  \\
11 & 1.5 & 0.75 & 0.7 & 7  & 51.11  & 3  & 1.11\%  \\
12 & 1.5 & 0.75 & 1 & 10  & 6.61  & 0  & -  \\
13 & 1.5 & 1 & 0.2 & 0  & -  & 10  & 36.17\%  \\
14 & 1.5 & 1 & 0.5 & 2  & 215.34  & 8  & 6.29\%  \\
15 & 1.5 & 1 & 0.7 & 9  & 26.63  & 1  & 3.1\%  \\
16 & 1.5 & 1 & 1 & 10  & 6.21  & 0  & -  \\
17 & 1.8 & 0.25 & 0.2 & 10  & 8.97  & 0  & -  \\
18 & 1.8 & 0.25 & 0.5 & 10  & 5.2  & 0  & -  \\
19 & 1.8 & 0.25 & 0.7 & 10  & 2.99  & 0  & -  \\
20 & 1.8 & 0.25 & 1 & 10  & 1.91  & 0  & -  \\
21 & 1.8 & 0.5 & 0.2 & 1  & 551.34  & 9  & 14.04\%  \\
22 & 1.8 & 0.5 & 0.5 & 9  & 126.62  & 1  & 2.68\%  \\
23 & 1.8 & 0.5 & 0.7 & 10  & 17.43  & 0  & -  \\
24 & 1.8 & 0.5 & 1 & 10  & 5.5  & 0  & -  \\
25 & 1.8 & 0.75 & 0.2 & 0  & -  & 10  & 30.75\%  \\
26 & 1.8 & 0.75 & 0.5 & 7  & 236.67  & 3  & 2.65\%  \\
27 & 1.8 & 0.75 & 0.7 & 10  & 18.03  & 0  & -  \\
28 & 1.8 & 0.75 & 1 & 10  & 6.39  & 0  & -  \\
29 & 1.8 & 1 & 0.2 & 0  & -  & 10  & 38.34\%  \\
30 & 1.8 & 1 & 0.5 & 0  & -  & 10  & 6.17\%  \\
31 & 1.8 & 1 & 0.7 & 8  & 74.08  & 2  & 5.4\%  \\
32 & 1.8 & 1 & 1 & 10  & 10.13  & 0  & -  \\
33 & 2.1 & 0.25 & 0.2 & 10  & 5.98  & 0  & -  \\
34 & 2.1 & 0.25 & 0.5 & 10  & 3.65  & 0  & -  \\
35 & 2.1 & 0.25 & 0.7 & 10  & 2.82  & 0  & -  \\
36 & 2.1 & 0.25 & 1 & 10  & 1.54  & 0  & -  \\
37 & 2.1 & 0.5 & 0.2 & 3  & 327.01  & 7  & 12.39\%  \\
38 & 2.1 & 0.5 & 0.5 & 9  & 26.93  & 1  & 2.38\%  \\
39 & 2.1 & 0.5 & 0.7 & 10  & 10.6  & 0  & -  \\
40 & 2.1 & 0.5 & 1 & 10  & 4.98  & 0  & -  \\
41 & 2.1 & 0.75 & 0.2 & 0  & -  & 10  & 26.76\%  \\
42 & 2.1 & 0.75 & 0.5 & 6  & 72.76  & 4  & 5.06\%  \\
43 & 2.1 & 0.75 & 0.7 & 9  & 69.47  & 1  & 1.23\%  \\
44 & 2.1 & 0.75 & 1 & 10  & 6.67  & 0  & -  \\
45 & 2.1 & 1 & 0.2 & 0  & -  & 10  & 33.59\%  \\
46 & 2.1 & 1 & 0.5 & 3  & 446.55  & 7  & 6.81\%  \\
47 & 2.1 & 1 & 0.7 & 7  & 140.71  & 3  & 3.57\%  \\
48 & 2.1 & 1 & 1 & 10  & 7.15  & 0  & -  \\

\end{tabular}
    \caption{Performance of the compact reformulation $(\FormAdj)$ for $\Gamma=5$.}
    \label{tab:G5}
\end{table}

\begin{table}[H]
    \centering
    \footnotesize
    \begin{tabular}{r r r r@{\hspace{1cm}}c r c r}
% res_adjustable_G7
class & NC & RF & RS & solved & time(s) & unsolved & gap \\
\hline
1 & 1.5 & 0.25 & 0.2 & 10  & 66.63  & 0  & -  \\
2 & 1.5 & 0.25 & 0.5 & 10  & 12.67  & 0  & -  \\
3 & 1.5 & 0.25 & 0.7 & 10  & 4.66  & 0  & -  \\
4 & 1.5 & 0.25 & 1 & 10  & 3.77  & 0  & -  \\
5 & 1.5 & 0.5 & 0.2 & 0  & -  & 10  & 17.86\%  \\
6 & 1.5 & 0.5 & 0.5 & 7  & 186.02  & 3  & 3.55\%  \\
7 & 1.5 & 0.5 & 0.7 & 10  & 25.1  & 0  & -  \\
8 & 1.5 & 0.5 & 1 & 10  & 6.53  & 0  & -  \\
9 & 1.5 & 0.75 & 0.2 & 0  & -  & 10  & 30.86\%  \\
10 & 1.5 & 0.75 & 0.5 & 3  & 269.24  & 7  & 6.02\%  \\
11 & 1.5 & 0.75 & 0.7 & 7  & 149.01  & 3  & 2.16\%  \\
12 & 1.5 & 0.75 & 1 & 10  & 10.02  & 0  & -  \\
13 & 1.5 & 1 & 0.2 & 0  & -  & 10  & 38.03\%  \\
14 & 1.5 & 1 & 0.5 & 1  & 45.89  & 9  & 7.24\%  \\
15 & 1.5 & 1 & 0.7 & 9  & 39.59  & 1  & 5.17\%  \\
16 & 1.5 & 1 & 1 & 10  & 9.48  & 0  & -  \\
17 & 1.8 & 0.25 & 0.2 & 10  & 16.87  & 0  & -  \\
18 & 1.8 & 0.25 & 0.5 & 10  & 5.01  & 0  & -  \\
19 & 1.8 & 0.25 & 0.7 & 10  & 3.88  & 0  & -  \\
20 & 1.8 & 0.25 & 1 & 10  & 2.39  & 0  & -  \\
21 & 1.8 & 0.5 & 0.2 & 1  & 847.24  & 9  & 15.06\%  \\
22 & 1.8 & 0.5 & 0.5 & 8  & 231.58  & 2  & 1.9\%  \\
23 & 1.8 & 0.5 & 0.7 & 10  & 18.29  & 0  & -  \\
24 & 1.8 & 0.5 & 1 & 10  & 8.74  & 0  & -  \\
25 & 1.8 & 0.75 & 0.2 & 0  & -  & 10  & 32.62\%  \\
26 & 1.8 & 0.75 & 0.5 & 6  & 68.88  & 4  & 2.75\%  \\
27 & 1.8 & 0.75 & 0.7 & 10  & 23.98  & 0  & -  \\
28 & 1.8 & 0.75 & 1 & 10  & 12.48  & 0  & -  \\
29 & 1.8 & 1 & 0.2 & 0  & -  & 10  & 39.47\%  \\
30 & 1.8 & 1 & 0.5 & 0  & -  & 10  & 7.24\%  \\
31 & 1.8 & 1 & 0.7 & 8  & 47.78  & 2  & 8.07\%  \\
32 & 1.8 & 1 & 1 & 10  & 13.33  & 0  & -  \\
33 & 2.1 & 0.25 & 0.2 & 10  & 8.95  & 0  & -  \\
34 & 2.1 & 0.25 & 0.5 & 10  & 4.86  & 0  & -  \\
35 & 2.1 & 0.25 & 0.7 & 10  & 4.05  & 0  & -  \\
36 & 2.1 & 0.25 & 1 & 10  & 2.18  & 0  & -  \\
37 & 2.1 & 0.5 & 0.2 & 4  & 624.37  & 6  & 14.42\%  \\
38 & 2.1 & 0.5 & 0.5 & 10  & 44.43  & 0  & -  \\
39 & 2.1 & 0.5 & 0.7 & 10  & 18.27  & 0  & -  \\
40 & 2.1 & 0.5 & 1 & 10  & 7.03  & 0  & -  \\
41 & 2.1 & 0.75 & 0.2 & 0  & -  & 10  & 26.71\%  \\
42 & 2.1 & 0.75 & 0.5 & 6  & 98.71  & 4  & 4.96\%  \\
43 & 2.1 & 0.75 & 0.7 & 8  & 71.07  & 2  & 1.11\%  \\
44 & 2.1 & 0.75 & 1 & 10  & 9.38  & 0  & -  \\
45 & 2.1 & 1 & 0.2 & 0  & -  & 10  & 34.48\%  \\
46 & 2.1 & 1 & 0.5 & 3  & 589.12  & 7  & 7.56\%  \\
47 & 2.1 & 1 & 0.7 & 6  & 197.8  & 4  & 2.67\%  \\
48 & 2.1 & 1 & 1 & 10  & 8.61  & 0  & -  \\
\end{tabular}
    \caption{Performance of the compact reformulation $(\FormAdj)$ for $\Gamma=7$.}
    \label{tab:G7}
\end{table}

\bibliographystyle{plainnat}
\bibliography{anchorRCPSP}

\begin{thebibliography}{19}
\providecommand{\natexlab}[1]{#1}
\providecommand{\url}[1]{\texttt{#1}}
\expandafter\ifx\csname urlstyle\endcsname\relax
  \providecommand{\doi}[1]{doi: #1}\else
  \providecommand{\doi}{doi: \begingroup \urlstyle{rm}\Url}\fi

\bibitem[Artigues et~al.(2003)Artigues, Michelon, and Reusser]{Artigues2003}
Christian Artigues, Philippe Michelon, and St{\'{e}}phane Reusser.
\newblock Insertion techniques for static and dynamic resource-constrained
  project scheduling.
\newblock \emph{Eur. J. Oper. Res.}, 149\penalty0 (2):\penalty0 249--267, 2003.

\bibitem[Artigues et~al.(2008)Artigues, Demassey, and
  Neron]{ArtiguesDemasseyNeronRCPSP}
Christian Artigues, Sophie Demassey, and Emmanuel Neron.
\newblock \emph{{Resource-Constrained Project Scheduling: Models, Algorithms,
  Extensions and Applications}}.
\newblock {ISTE/Wiley}, 2008.

\bibitem[Artigues et~al.(2013)Artigues, Leus, and
  Nobibon]{ArtiguesLeusTallaNobibon2013}
Christian Artigues, Roel Leus, and Fabrice~Talla Nobibon.
\newblock Robust optimization for resource-constrained project scheduling with
  uncertain activity durations.
\newblock \emph{Flexible Services and Manufacturing Journal}, 25:\penalty0
  175--205, 2013.
\newblock \doi{10.1007/s10696-012-9147-2}.

\bibitem[Ben{-}Tal et~al.(2004)Ben{-}Tal, Goryashko, Guslitzer, and
  Nemirovski]{BGGNAdjustable}
Aharon Ben{-}Tal, A.~P. Goryashko, E.~Guslitzer, and Arkadi Nemirovski.
\newblock Adjustable robust solutions of uncertain linear programs.
\newblock \emph{Mathematical Programming}, 99\penalty0 (2):\penalty0 351--376,
  2004.
\newblock \doi{10.1007/s10107-003-0454-y}.

\bibitem[Bendotti et~al.(2017)Bendotti, Chr{\'{e}}tienne, Fouilhoux, and
  Quilliot]{CPMEJOR}
Pascale Bendotti, Philippe Chr{\'{e}}tienne, Pierre Fouilhoux, and Alain
  Quilliot.
\newblock Anchored reactive and proactive solutions to the {C}{P}{M}-scheduling
  problem.
\newblock \emph{European Journal of Operational Research}, 261\penalty0
  (1):\penalty0 67--74, 2017.
\newblock \doi{10.1016/j.ejor.2017.02.007}.

\bibitem[Bendotti et~al.(2019)Bendotti, Chr{\'e}tienne, Fouilhoux, and
  Pass-Lanneau]{AnchRobPSPHal}
Pascale Bendotti, Philippe Chr{\'e}tienne, Pierre Fouilhoux, and Ad{\`e}le
  Pass-Lanneau.
\newblock {The Anchor-Robust Project Scheduling Problem}.
\newblock May 2019.
\newblock URL \url{https://hal.archives-ouvertes.fr/hal-02144834}.

\bibitem[Bendotti et~al.(2020)Bendotti, Chr\'etienne, Fouilhoux, and
  Pass-Lanneau]{ISCOAnchResched}
Pascale Bendotti, Philippe Chr\'etienne, Pierre Fouilhoux, and Ad\`ele
  Pass-Lanneau.
\newblock Anchored rescheduling problems under generalized precedence
  constraints.
\newblock In Mourad Ba\"iou, Bernard Gendron, Oktay G\"unl\"uk, and Ali~Ridha
  Mahjoub, editors, \emph{Combinatorial Optimization. ISCO 2020}, volume 12176
  of \emph{Lecture Notes in Computer Science}, 2020.
\newblock \doi{doi.org/10.1007/978-3-030-53262-8\_13}.

\bibitem[Bertsimas and Sim(2004)]{PriceBS}
Dimitris Bertsimas and Melvyn Sim.
\newblock The price of robustness.
\newblock \emph{Operations Research}, 52:\penalty0 35--53, 2004.
\newblock \doi{10.1287/opre.1030.0065}.

\bibitem[Bold and Goerigk(2020)]{BoldGoerigk2020}
Matthew Bold and Marc Goerigk.
\newblock A compact reformulation of the two-stage robust resource-constrained
  project scheduling problem.
\newblock \emph{CoRR}, abs/2004.06547, 2020.
\newblock URL \url{https://arxiv.org/abs/2004.06547}.

\bibitem[Bruni et~al.(2016)Bruni, Di~Puglia~Pugliese, Beraldi, and
  Guerriero]{Bruni2016}
Maria Bruni, Luigi Di~Puglia~Pugliese, Patrizia Beraldi, and Francesca
  Guerriero.
\newblock An adjustable robust optimization model for the resource-constrained
  project scheduling problem with uncertain activity durations.
\newblock \emph{Omega}, 09 2016.
\newblock \doi{10.1016/j.omega.2016.09.009}.

\bibitem[Bruni et~al.(2018)Bruni, Pugliese, Beraldi, and Guerriero]{Bruni2018}
Maria~Elena Bruni, Luigi Di~Puglia Pugliese, Patrizia Beraldi, and Francesca
  Guerriero.
\newblock A computational study of exact approaches for the adjustable robust
  resource-constrained project scheduling problem.
\newblock \emph{Comput. Oper. Res.}, 99:\penalty0 178--190, 2018.
\newblock \doi{10.1016/j.cor.2018.06.016}.

\bibitem[Hazır and Ulusoy(2020)]{HazirUlusoy2020}
Öncü Hazır and Gündüz Ulusoy.
\newblock A classification and review of approaches and methods for modeling
  uncertainty in projects.
\newblock \emph{International Journal of Production Economics}, 223:\penalty0
  107522, 2020.
\newblock ISSN 0925-5273.
\newblock \doi{https://doi.org/10.1016/j.ijpe.2019.107522}.

\bibitem[Herroelen and Leus(2002)]{HerroelenLeusProjectSchedulingUncertainty}
Willy Herroelen and Roel Leus.
\newblock Project scheduling under uncertainty: Survey and research potentials.
\newblock \emph{European Journal of Operational Research}, 165:\penalty0
  289--306, 2002.
\newblock \doi{10.1016/j.ejor.2004.04.002}.

\bibitem[Herroelen and Leus(2004)]{HerroelenStablePreschedule}
Willy Herroelen and Roel Leus.
\newblock The construction of stable project baseline schedules.
\newblock \emph{European Journal of Operational Research}, 156\penalty0
  (3):\penalty0 550--565, 2004.
\newblock \doi{10.1016/S0377-2217(03)00130-9}.

\bibitem[Kolisch and Hartmann(1999)]{KolischHartmannRCPSPHeur}
Rainer Kolisch and S{\"o}nke Hartmann.
\newblock \emph{Heuristic Algorithms for the Resource-Constrained Project
  Scheduling Problem: Classification and Computational Analysis}, pages
  147--178.
\newblock Springer US, Boston, MA, 1999.
\newblock ISBN 978-1-4615-5533-9.
\newblock \doi{10.1007/978-1-4615-5533-9\_7}.

\bibitem[Kon{\'e} et~al.(2013)Kon{\'e}, Artigues, Lopez, and Mongeau]{Kone2013}
Oumar Kon{\'e}, Christian Artigues, Pierre Lopez, and Marcel Mongeau.
\newblock {Comparison of mixed integer linear programming models for the
  resource-constrained project scheduling problem with consumption and
  production of resources}.
\newblock \emph{{Flexible Services and Manufacturing Journal}}, 25\penalty0
  (1-2):\penalty0 24--47, 2013.
\newblock \doi{10.1007/s10696-012-9152-5}.

\bibitem[Minoux(2007)]{Minoux2007}
Michel Minoux.
\newblock Models and algorithms for robust {P}{E}{R}{T} scheduling with
  time-dependent task durations.
\newblock \emph{Vietnam Journal of Mathematics}, 35, 01 2007.

\bibitem[Soyster(1973)]{Soyster}
Allen~L. Soyster.
\newblock Technical note -- convex programming with set-inclusive constraints
  and applications to inexact linear programming.
\newblock \emph{Operations Research}, 21\penalty0 (5):\penalty0 1154--1157,
  1973.
\newblock \doi{10.1287/opre.21.5.1154}.

\bibitem[Zeng and Zhao(2013)]{ZZ}
Bo~Zeng and Long Zhao.
\newblock Solving two-stage robust optimization problems using a
  column-and-constraint generation method.
\newblock \emph{Operations Research Letters}, 41\penalty0 (5):\penalty0 457 --
  461, 2013.
\newblock ISSN 0167-6377.
\newblock \doi{https://doi.org/10.1016/j.orl.2013.05.003}.

\end{thebibliography}

\end{document}